\documentclass[12pt]{amsart}
\usepackage{amssymb,amsmath}
\usepackage{geometry}    
\usepackage{mathrsfs}

\usepackage{vmargin}
\setmargrb{1in}{1in}{1in}{1in}

\newtheorem{theorem}{Theorem}[section]
\newtheorem{lemma}[theorem]{Lemma}
\newtheorem{proposition}[theorem]{Proposition}
\newtheorem{corollary}[theorem]{Corollary}
\theoremstyle{definition}
\newtheorem{definition}[theorem]{Definition}
\newtheorem{remark}[theorem]{Remark}

\begin{document}

\title{Uniform Diophantine approximation and best approximation polynomials}

\author{Johannes Schleischitz} 

\thanks{Research supported by the Schr\"odinger scholarship J 3824 of the Austrian Science Fund FWF.\\
Department of Mathematics and Statistics, University of Ottawa, Canada  \\ 
johannes.schleischitz@univie.ac.at}

\begin{abstract}
Let $\zeta$ be a real transcendental number.
We introduce a new method to find upper bounds for the classical 
exponent $\widehat{w}_{n}(\zeta)$ concerning uniform
polynomial approximation. Our method is
based on the
parametric geometry of numbers introduced by Schmidt and Summerer,
and transference of the original
approximation problem in dimension $n$ to suitable
higher dimensions.
For large $n$, we can provide an unconditional bound of order
$\widehat{w}_{n}(\zeta)\leq 2n-2+o(1)$. While
this improves the bound of order $2n-\frac{3}{2}+o(1)$ due to Bugeaud and the 
author, it is unfortunately slightly weaker than what can be
obtained when incorporating a recently proved conjecture
of Schmidt and Summerer. 
However, the method also enables us to
establish significantly stronger conditional bound upon a certain
presumably weak
assumption on the structure of the best approximation polynomials.
Thereby we provide serious evidence that the exponent should be
reasonably smaller than the known upper bounds.
\end{abstract}

\maketitle

{\footnotesize 
{\em Keywords}: uniform approximation, geometry of numbers, successive minima \\
Math Subject Classification 2010: 11H06, 11J13, 11J82}


\section{An exponent of uniform Diophatine approximation}

\subsection{Introduction} \label{intro}

A classical problem in Diophantine approximation is 
to study the approximation to a real number $\zeta$ by algebraic
real numbers $\alpha$ of bounded degree. 
For example, the famous
Wirsing problem~\cite{wirsing} asks if for any real number $\zeta$ 
and any $\epsilon>0$ the estimate
\[
\vert \alpha-\zeta\vert \leq H(\alpha)^{-n-1+\epsilon}
\]
has infinitely many solutions in real algebraic numbers
$\alpha$ of degree at most $n$. 
Here $H(\alpha)=H(P_{\alpha})$ denotes the naive height of the minimal polynomial $P_{\alpha}$
of $\alpha$ over $\mathbb{Z}$ with coprime coefficients, 
that is the maximum modulus among the coefficients of $P_{\alpha}$.
In case $n=1$ this is an immediate consequence
of Dirichlet's Theorem.
It is further known that the answer to Wirsing's problem
is positive for $n=2$, see~\cite{davs},
as well as for any $n$ and Lebesgue almost all real numbers.
The problem of making $\vert \zeta-\alpha\vert$
small naturally
leads to the problem of making $\vert P(\zeta)\vert$  small
among integer polynomials $P$, with the same degree and (naive)
height restrictions as for $\alpha$.
Indeed with the above notation it is straight forward
to show that $\vert P_{\alpha}(\zeta)\vert\ll_{n,\zeta} 
 H(P_{\alpha})\cdot \vert \alpha-\zeta\vert$,
and conversely $\vert \alpha-\zeta\vert\ll_{n,\zeta} 
 H(P_{\alpha})^{n}\cdot \vert P_{\alpha}(\zeta)\vert$ 
can be found in~\cite[Lemma~A8]{bugbuch}. 
This connection motivates to define the 
classical exponents of approximation $w_{n}(\zeta)$
and $\widehat{w}_{n}(\zeta)$. They are given as
the supremum of $w\in{\mathbb{R}}$ for which the system 
\begin{equation}  \label{eq:w}
H(P) \leq X, \qquad  0<\vert P(\zeta)\vert \leq X^{-w},  
\end{equation}
has a solution in an integer polynomial
$P$ of degree at most $n$ for arbitrarily large $X$, and all large $X$, respectively. The exponents
$w_{n}(\zeta)$ already date back to Mahler~\cite{maler}.
The most fundamental property of these exponents that stems
from Dirichlet's Theorem is
\begin{equation} \label{eq:wmono}
w_{n}(\zeta)\geq \widehat{w}_{n}(\zeta)\geq n.
\end{equation}
Moreover we want to mention the obvious relations
\begin{equation} \label{eq:wmonos}
w_{1}(\zeta)\leq w_{2}(\zeta)\leq \cdots, 
\qquad \widehat{w}_{1}(\zeta)\leq \widehat{w}_{2}(\zeta)\leq \cdots. 
\end{equation}
For $\zeta$ real algebraic of degree $d$, Schmidt's Subspace 
Theorem yields
\begin{equation} \label{eq:alg}
w_{n}(\zeta)= \widehat{w}_{n}(\zeta)= \min\{ d-1, n\}.
\end{equation}
In this paper we investigate the uniform exponents
$\widehat{w}_{n}(\zeta)$ for transcendental $\zeta$.
Davenport and Schmidt~\cite[Theorem~2b]{davsh}, although not
using our notation, established
the upper bound $\widehat{w}_{n}(\zeta)\leq 2n-1$
for any transcendental real number $\zeta$. This 
appears a little surprising in some sense,
as for $\mathbb{Q}$-linearly independent vectors
$\underline{\zeta}$,
the corresponding uniform exponent can attain any value 
in $[n,+\infty]$ when $n\geq 2$. Concrete examples of 
$\underline{\zeta}\in\mathbb{R}^{n}$
with prescribed uniform exponent in the 
allowed interval can be readily derived
from~\cite[Theorem~2.5]{ich1} with a suitable choice of the occurring parameters $\eta_{i}$, and
Mahler's theorem on polar reciprocal bodies. The coordinates 
of $\underline{\zeta}$
can be chosen Liouville numbers in Cantor's middle third set.
The deep theorem by Roy~\cite{royann} yields 
another non-constructive proof.
For $n=2$, from~\cite[Theorem~1a]{davsh}
concerning the dual simultaneous approximation problem
and Jarn\'ik's identity~\cite{jarid} the estimate 
$\widehat{w}_{2}(\zeta)\leq 2\cdot 2-1=3$ can be improved to 
\begin{equation} \label{eq:najo}
\widehat{w}_{2}(\zeta)\leq \frac{3+\sqrt{5}}{2}= 2.6180\ldots.
\end{equation}
Surprisingly this was shown to be sharp by Roy~\cite{royyy}.
Numbers with $\widehat{w}_{2}(\zeta)>2$ and the 
set of values 
taken by $\widehat{w}_{2}(\zeta)$ when
$\zeta$ runs through the transcendental real numbers, have been intensely investigated, see for 
example~\cite{buglau},~\cite{f1},~\cite{f2},~\cite{f3},~\cite{rroy},~\cite{royexp}. Much less is known for $n\geq 3$.
It is still an open problem to decide
if $\widehat{w}_{n}(\zeta)>n$ can 
occur, and
the upper bound $2n-1$ had not been improved
for almost 50 years until recently it was shown that
\begin{equation} \label{eq:buschlei}
\widehat{w}_{n}(\zeta)\leq n-\frac{1}{2}+\sqrt{n^2-2n+\frac{5}{4}}, \qquad\qquad n\geq 2,
\end{equation}
in \cite{buschl}. This is sharp 
again for $n=2$, and of order $2n-\frac{3}{2}+\varepsilon_{n}$ 
with $\varepsilon_{n}>0$ of order $O(n^{-1})$ for larger $n$. 
For $n=3$, the confirmation of a special case of a conjecture of Schmidt 
and Summerer~\cite{sums} on the quotient $w_{n}(\zeta)/\widehat{w}_{n}(\zeta)$
implied
the stronger bound
\begin{equation} \label{eq:schleibu}
\widehat{w}_{3}(\zeta)\leq 3+\sqrt{2}=4.4142\ldots.
\end{equation}
as obtained in the same paper~\cite{buschl}. Due to a 
surprising very recent proof of the conjecture in full
generality announced by Marnat and Moshchevitin~\cite{mamo}, 
in fact for
$n\geq 4$ the bound \eqref{eq:buschlei} can
be improved as well.
As pointed out in~\cite{ichglasgow} we can deduce
\begin{equation} \label{eq:frwe}
\widehat{w}_{n}(\zeta)\leq \max\{2n-2,w(n)\}
\end{equation}
for any real $\zeta$, where $w=w(n)$ as the unique
solution of the polynomial identity
\begin{equation} \label{eq:schleim}
\frac{(n-1)w}{w-n}-w+1= \left( \frac{n-1}{w-n}\right)^{n}
\end{equation}
in the interval $[n,2n-1)$. Precisely for $n\geq 10$, we have $w(n)<2n-2$ and \eqref{eq:frwe} becomes
$\widehat{w}_{n}(\zeta)\leq 2n-2$. 
The asymptotic formula 
\begin{equation} \label{eq:asyb}
w(n)= 2n-C+\epsilon_{n}, \qquad\quad
C=2.2564\ldots,
\end{equation}
with $\epsilon_{n}>0$ which tends to $0$ as $n\to\infty$
was further shown
in~\cite[Theorem~3.1]{ichglasgow}. As the bounds
within the now settled Schmidt-Summerer Conjecture
is known to be optimal, the bound 
in the right hand
side in \eqref{eq:asyb} appears to be the limit of this method.

\subsection{A new method}
The purpose of the current paper is to 
develop a new method to further investigate 
upper bounds for the exponent $\widehat{w}_{n}(\zeta)$. Unfortunately, at present we are not able
to deduce from it an unconditional improvement of
the currently best known
bound in \eqref{eq:frwe} for any $n$. However,
there are some benefits of our new method. Firstly,
we can provide a refinement of \eqref{eq:buschlei}
without usage of explicit bounds for the quotient
$w_{n}(\zeta)/\widehat{w}_{n}(\zeta)$, in particular we do not
apply the above mentioned deep
Schmidt-Summerer Conjecture. 
Furthermore, we will
present conditional results with significantly better bounds
for the exponent $\widehat{w}_{n}(\zeta)$. They
motivate further investigation of the underlying method of 
this paper to find the correct magnitude of the exponent. 
We start with our unconditional result.

\begin{theorem} \label{thm}
Let $n$ be a positive integer and $\zeta$ be a real number. Then 
\begin{equation} \label{eq:neu}
\widehat{w}_{n}(\zeta)\leq \theta_{n}:=
\frac{3(n-1)+\sqrt{n^{2}-2n+5}}{2}.
\end{equation}
\end{theorem}

For $n=2$, the value $\theta_{n}$ again coincides with 
the sharp bound \eqref{eq:buschlei}. 
For $n=3$, it becomes the bound in \eqref{eq:schleibu}.
The latter is surprising, since 
the method in the proof of Theorem~\ref{thm}
differs significantly from the one in~\cite{buschl}.
The proof in~\cite{buschl} partly relied on deep results of Schmidt and Summerer~\cite{ssch},\cite{sums}
concerning the minimum gap between the values $w_{n}(\zeta)$ 
and $\widehat{w}_{n}(\zeta)$. Although our proof 
of Theorem~\ref{thm} employs
fundamental principles of
the parametric geometry developed by Schmidt 
and Summerer in~\cite{ss},
our approach does not require the deeper
results from~\cite{ssch},\cite{sums} nor the
proof of their conjecture in~\cite{mamo}. 
The two different proofs provide
a cautious indication that \eqref{eq:schleibu}
could be optimal (or at least close).
For $n\geq 4$, the estimate is of order $2n-2+o(1)$ as $n\to\infty$.
As indicated above, it improves \eqref{eq:buschlei} but
is slightly weaker than \eqref{eq:frwe}. 
For small $n$ we state the numerical values 
\[
\theta_{4}= 6.3028\ldots, 
\qquad \theta_{5}= 8.2361\ldots, 
\qquad \theta_{10}= 18.1098\ldots,
\]
which should be compared to $6.2875\ldots, 8.2010\ldots$ and $10$
respectively from \eqref{eq:frwe}.

\section{Conditional results} \label{newr}

%
%

\subsection{A conditional bound} \label{n}
Our second main result of this paper Theorem~\ref{gutsatz}
below yields a significant
improvement on the upper bounds in \eqref{eq:schleim}
conditional on some property of
the structure of best approximation polynomials.
We need some preparation before we can state it. First we
recall the notion of best approximation polynomials.

\begin{definition} \label{newdef}
For $n\geq 1$ and integer and $\zeta$ a real number, an
integer polynomial $P$ of degree at most $n$ will be called
{\em best approximation polynomial associated to $(n,\zeta)$}
if it minimizes $\vert P(\zeta)\vert$ among all non identically 
zero integer polynomials of degree at most $n$ and height at most $H(P)$. Any pair $n,\zeta$ thus
gives rise to a uniquely determined sequence 
of best approximation polynomials. We denote it by 
$(P_{k}^{n,\zeta})_{k\geq 1}$.
\end{definition}

The definition implies that
\[
H(P_{1}^{n,\zeta})< H(P_{2}^{n,\zeta})< \cdots, \qquad\qquad
\vert P_{1}^{n,\zeta}(\zeta)\vert > \vert P_{2}^{n,\zeta}(\zeta)\vert > \cdots,
\]
and there is no non-zero 
integer polynomial $P\neq P_{i}^{n,\zeta}$ of 
degree at most $n$ and height
$H(P)\leq H(P_{i}^{n,\zeta})$ such that
$\vert P(\zeta)\vert < \vert P_{i-1}^{n,\zeta}(\zeta)\vert$. 
In Theorem~\ref{gutsatz} below
we show that a putative large value of $\widehat{w}_{n}(\zeta)$
for given $n\geq 1$ and transcendental real $\zeta$ 
implies a very biased structure of the best approximation 
polynomials $P_{k}^{n,\zeta}$ associated to $(n,\zeta)$. 
Let us write
\begin{equation} \label{eq:dingung}
P_{k}^{n,\zeta}(T)= h_{k,0}+h_{k,1}T+\cdots+h_{k,n}T^{n}, 
\qquad\qquad k\geq 1,\; h_{k,j}=h_{k,j}(n,\zeta)\in\mathbb{Z},
\end{equation}
and further let 
\[
\underline{h}_{k}^{n,\zeta}=(h_{k-1,0},\ldots,h_{k-1,n}, h_{k,0}, \ldots, h_{k,n},
h_{k+1,0}, \ldots, h_{k+1,n})\in\mathbb{Z}^{3n+3}, \qquad k\geq 2,
\]
be the vector consisting of the coordinates
of three consecutive best approximation
polynomials glued together.
Next for any even $n\geq 2$ we define a
$(3\frac{n}{2}\times 3\frac{n}{2})$-integer matrix 
$\Lambda_{n}=\Lambda_{n}(\underline{x})=\Lambda_{n}(x_{1},\ldots,x_{3n+3})$
with some circulant structure. For $n=2$, when inserting
$\underline{x}=\underline{h}_{k}^{2,\zeta}$ above, the
columns of 
$\Lambda_{2}(\underline{h}_{k}^{2,\zeta})\in\mathbb{Z}^{3\times 3}$ consist precisely of the coordinate
vectors of three consecutive best approximation polynomials
$P_{k-1}^{n,\zeta}(\zeta), P_{k}^{n,\zeta}(\zeta), 
P_{k+1}^{n,\zeta}(\zeta)$. 
For $n=4$ it is given by 
\[
\Lambda_{4}(\underline{h}_{k}^{4,\zeta})=
\begin{pmatrix} h_{k-1,0} & 0 & h_{k,0} & 0 & 
 h_{k+1,0} & 0 
 \\ 
h_{k-1,1} & h_{k-1,0} & h_{k,1} & 
 h_{k,0} & h_{k+1,1} & h_{k+1,0} \\   
 h_{k-1,2} & h_{k-1,1} & h_{k,2} & 
 h_{k,1} & h_{k+1,2} & h_{k+1,1} \\
 h_{k-1,3} & h_{k-1,2} & h_{k,3} & 
 h_{k,2} & h_{k+1,3} & h_{k+1,2} \\
 h_{k-1,4} & h_{k-1,3} & h_{k,4} & 
 h_{k,3} & h_{k+1,4} & h_{k+1,3} \\ 
 0 & h_{k-1,4} & 0 & h_{k,4} & 
 0 & h_{k+1,4} 
 \end{pmatrix}, \qquad k\geq 2.
\]
Similarly, for even $n\geq 6$ 
the matrix $\Lambda_{n}(\underline{h}_{k}^{n,\zeta})$ 
arises by putting the
vectors
\[
(h_{j,0},h_{j,1},\ldots,h_{j,n},0,\ldots,0)\in\mathbb{Z}^{3n/2},
 \qquad j\in\{k-1,k,k+1\},
\]
corresponding to $P_{k-1}^{n,\zeta}, P_{k}^{n,\zeta},
P_{k+1}^{n,\zeta}$ in the columns $1, \frac{n}{2}+1$ and
$n+1$ respectively, and shifting each modulo $3n/2$
from once up to $(\frac{n}{2}-1)$-times successively
to obtain the remaining $3\cdot (\frac{n}{2}-1)$ columns.
Define further
\[
\Phi_{n}(\underline{x})=
\det \Lambda_{n}(\underline{x}),
\]
which is a homogeneous polynomial in $3n+3$ variables
of total degree $3n+3$.

\begin{lemma}  \label{lemur}
Let $n\geq 2$ be an even integer and $\zeta$ be a transcendental real number.
Then for any $k\geq 2$ the following claims are equivalent.
\begin{itemize}
\item We have
\begin{equation} \label{eq:edingung}
\Phi_{n}(\underline{h}_{k}^{n,\zeta})\neq 0.
\end{equation}
\item The identity
\begin{equation}  \label{eq:bedingung}
A_{k}P_{k-1}^{n,\zeta}+B_{k}P_{k}^{n,\zeta}+
C_{k}P_{k+1}^{n,\zeta}\equiv 0
\end{equation}
in integer polynomials 
$A_{k}, B_{k}, C_{k}$ each of degree at most $\frac{n}{2}-1$
has only the trivial solution 
$A_{k}\equiv B_{k}\equiv C_{k}\equiv 0$.
\item The space spanned by 
\begin{equation} \label{eq:ingung}
\bigcup_{0\leq j\leq 
\frac{n}{2}-1}\left\{ T^{j}P_{k-1}^{n,\zeta},\; 
T^{j}P_{k}^{n,\zeta},\; T^{j}P_{k+1}^{n,\zeta}\right\}
\end{equation}
has full dimension $3\frac{n}{2}$, i.e. the polynomials
span the space
of polynomials of degree at most $3\frac{n}{2}-1$ isomorphic 
to $\mathbb{R}^{3n/2}$ in a direct sum 
\end{itemize}
\end{lemma}

The proof relies on elementary linear algebra considerations
and will be given in Section~\ref{proofs}.
If the equivalent conditions hold for some $k$,
it is not hard to conclude
that the three polynomials $P_{k-1}^{n,\zeta}, P_{k}^{n,\zeta}$
and $P_{k+1}^{n,\zeta}$ have no common factor. Unfortunately, 
for any pair among them this is not clear. If that were true 
we could remove
an inconvenient condition in Theorem~\ref{gutsatz} below,
our main result of this section.

\begin{theorem} \label{gutsatz}
Let $n$ be an even integer and 
$\zeta$ be a transcendental real number
for which the equivalent conditions of
Lemma~\ref{lemur} hold for infinitely many $k$.
For $n\geq 4$, we have
\begin{equation} \label{eq:wehave}
\widehat{w}_{n}(\zeta) \leq 2n-2.
\end{equation}
Moreover, if we assume $\widehat{w}_{n}(\zeta)>m=\frac{3}{2}n-1$ then we have
\begin{equation} \label{eq:derive}
\frac{w_{n}(\zeta)}{\widehat{w}_{n}(\zeta)}\geq \frac{2(\widehat{w}_{n}(\zeta)-n+1)}{n}.
\end{equation}
If $n\geq 2$ and
additionally 
\begin{equation} \label{eq:burt}
w_{n}(\zeta)> w_{n-1}(\zeta)
\end{equation}
holds, then in place of \eqref{eq:wehave}
we have the significantly stronger estimate
\begin{equation} \label{eq:rhs}
\widehat{w}_{n}(\zeta) \leq \sigma_{n}
:=\frac{2n-1+\sqrt{2n^2-2n+1}}{2}.
\end{equation}
\end{theorem}

The value $\sigma_{n}$ is just slightly smaller
than $(1+\frac{1}{\sqrt{2}})\cdot n$, 
hence \eqref{eq:rhs} is reasonably stronger than both the 
best known unconditional bound in
\eqref{eq:frwe}. In view of \eqref{eq:wmonos}, we
clearly obtain the same conditional asymptotic estimate
for odd $n$ as well. 
For subtle technical reasons our proof of \eqref{eq:rhs}
requires the unpleasant condition \eqref{eq:burt}, 
however we believe that it can be removed. Various
alternative conditions instead of \eqref{eq:burt} can be stated
that imply \eqref{eq:rhs} as well, for example that 
$P_{k}^{n,\zeta}$ and $P_{k+1}^{n,\zeta}$
have no common factor for those indices $k$ for which
the conditions of Lemma~\ref{lemur} are satisfied. 
Observe that even \eqref{eq:wehave}
is stronger than the conditional bound \eqref{eq:frwe} 
when $n\in\{4,6,8\}$, in contrast 
to \eqref{eq:neu} which is weaker than 
\eqref{eq:frwe} for any $n\geq 4$ as pointed out
in~Section~\ref{intro}.
The first few values $\sigma_{n}$
are numerically given as
\[
\sigma_{2}= 2.6180\ldots, \qquad \sigma_{4}= 6,
\qquad \sigma_{6}= 9.4051\ldots, \qquad \sigma_{8}=12.8150\ldots.
\]

We discuss how strong the assumption of Theorem~\ref{gutsatz} is.
If the condition fails, then $\Phi_{n}$ has to vanish 
at the vectors $\underline{h}_{k}^{n,\zeta}$ consisting of three
consecutive best approximation polynomials, for every large 
$k\geq k_{0}$. Hence all these vectors $\underline{h}_{k}^{n,\zeta},
k\geq k_{0}$,
must belong to some fixed collection of lower dimension submanifolds
of $\mathbb{R}^{3n+3}$. It seems natural to expect this
should not happen for any transcendental real $\zeta$. 
Moreover, for $n=2$ the condition can indeed be verified, 
and leads to another proof
of \eqref{eq:najo}. For $n=2$, we have $\frac{n}{2}-1=0$, and thus \eqref{eq:ingung} becomes 
that for any transcendental $\zeta$ and
infinitely many $k$, three consecutive 
quadratic best approximation
polynomials $P_{k-1}^{2,\zeta},P_{k}^{2,\zeta},P_{k+1}^{2,\zeta}$ are linearly independent. Indeed,
it has been known for a long time that this holds true. More 
generally for any $n\geq 2$ and any $\mathbb{Q}$-linearly independent set
$\underline{\zeta}=\{ 1,\zeta_{1},\ldots, \zeta_{n}\}$,
there exist infinitely
many $k$ such that three consecutive similarly defined best approximation linear forms
$L_{k-1}^{n,\underline{\zeta}},L_{k}^{n,\underline{\zeta}},L_{k+1}^{n,\underline{\zeta}}$ are linearly independent. 
See for example in
the remark on page~77 in~\cite{ss}. 
Equivalently, condition \eqref{eq:bedingung}
of Theorem~\ref{gutsatz} cannot fail with
constant polynomials $A_{k}, B_{k}, C_{k}$ for all large $k$.
However, the analogous claim
concerning four (or more) consecutive linear forms
$L_{k-1}^{n,\underline{\zeta}},L_{k}^{n,\underline{\zeta}},L_{k+1}^{n,\underline{\zeta}},L_{k+2}^{n,\underline{\zeta}}$
is false for any $n\geq 4$, as shown 
by Moshchevitin~\cite{mosh}. 
It remains unknown whether there exist counterexamples of the form
$\underline{\zeta}=(1,\zeta,\zeta^{2},\ldots,\zeta^{n})$ 
we are concerned with.

We finally compare \eqref{eq:derive} with
the bound resulting from the recently 
verified Schmidt-Summerer conjecture. As essentially 
shown in~\cite{ichglasgow}
the latter can be stated by the implicit inequality
\begin{equation} \label{eq:schs}
\left(\frac{w_{n}(\zeta)}{\widehat{w}_{n}(\zeta)}\right)^{n}\geq w_{n}(\zeta)-\widehat{w}_{n}(\zeta)+1.
\end{equation}
Here we have to take the solution 
with $w_{n}(\zeta)\neq \widehat{w}_{n}(\zeta)$ unless
$w_{n}(\zeta)=\widehat{w}_{n}(\zeta)=n$.
A small
computation shows that as
soon as $\widehat{w}_{n}(\zeta)>m+\mu+o(1)$, where
$\mu=0.38\ldots$ is some real number, \eqref{eq:derive} 
is stronger than \eqref{eq:schs}. Hence, 
we obtain a potentially large range of
parameters $[\frac{3}{2}n-O(1), (1+\frac{1}{\sqrt{2}})n-O(1)]$
where \eqref{eq:schs} is refined. One may compare this to 
Remark~\ref{dierem} below. 

\subsection{Variations of Theorem~\ref{gutsatz}}

It is clear
that $\frac{n}{2}-1$ is the smallest integer that admits the conditions \eqref{eq:bedingung} and \eqref{eq:ingung}, 
hence \eqref{eq:rhs} 
seems to be the limit of the method.
If we ask for restriction by some larger
integer $\frac{3n}{2}-1<m<2n-1$, we instead obtain 
some bound between $\theta_{n}$ and $\sigma_{n}$ from
\eqref{eq:neu} and \eqref{eq:rhs} respectively, which
increases as $m$ grows. Concretely, we will establish the
generalization Theorem~\ref{cumbersome} of Theorem~\ref{gutsatz}
below, based on generalizations of condition \eqref{eq:ingung}.
We introduce some more notation.

\begin{definition} \label{dadp}
Let $n\geq 2$ be an integer, $\zeta$ be a transcendental real
number and $(P_{i}^{n,\zeta})_{i\geq 1}$ the sequence
of best apporoximation polynomials associated to $n,\zeta$. 
Say $d_{i}\leq n$ denotes the degree
of $P_{i}^{n,\zeta}$. For any integer $m\geq n$, let
\begin{equation} \label{eq:zingung}
\mathscr{A}_{i,m}^{n,\zeta}:=
\bigcup_{0\leq j\leq m-d_{i}} 
\{ T^{j}\cdot P_{i}^{n,\zeta}(T)\},\qquad 
\qquad i\geq 1,
\end{equation}
and for any $k\geq 2$, consider the unions
of three consecutive sets
\begin{equation} \label{eq:zwanzig}
\mathscr{B}_{k,m}^{n,\zeta}:= 
\mathscr{A}_{k-1,m}^{n,\zeta}\cup \mathscr{A}_{k,m}^{n,\zeta}\cup \mathscr{A}_{k+1,m}^{n,\zeta}. 
\end{equation}
We say the pair $(n,\zeta)$ has the property $span(m)$, if for infinitely many $k$
the set $\mathscr{B}_{k,m}^{n,\zeta}$
spans the $(m+1)$-dimensional space of polynomials of 
degree at most $m$. Furthermore we say $(n,\zeta)$ has the 
property $\widetilde{span}(m)$ if satisfies
$span(m)$ and additionally $P_{k-1}^{n,\zeta}$ and $P_{k}^{n,\zeta}$ within the definition can be chosen without common factor. 
Moreover, we
denote by $\Psi(n,\zeta)$ the smallest integer $m$ 
with the property $span(m)$, and $\widetilde{\Psi}(n,\zeta)$ the 
smallest integer $m$ 
with the property $\widetilde{span}(m)$. 
Further let
\[
\Psi(n)= \max_{\zeta\in \mathbb{R}\setminus{\overline{\mathbb{Q}}}}\Psi(n,\zeta), \qquad\quad
\widetilde{\Psi}(n)= \max_{\zeta\in \mathbb{R}\setminus{\overline{\mathbb{Q}}}}\widetilde{\Psi}(n,\zeta). 
\]
\end{definition}

\begin{remark} \label{hirsch}
The distinction between $\Psi$ and
$\widetilde{\Psi}$ is subtle.
Davenport and Schmidt showed that for any 
$n\geq 2$ and any transcendental real 
$\zeta$,
two successive best approximation
polynomials $P_{j-1}^{n,\zeta},P_{j}^{n,\zeta}$
are coprime infinitely often~\cite[Theorem~2b]{davsh}, however
given $span(m)$
we do not know whether this holds for those values $j=k$ with the
property that $\mathscr{B}_{k,m}^{n,\zeta}$ spans
the space of polynomials of degree at most $m$. Even for
$n=2$ this seems not obvious. The stronger 
property  $\widetilde{span}(m)$
guarantees that the space spanned by 
$\mathscr{A}_{k-1,m}^{n,\zeta}\cup \mathscr{A}_{k,m}^{n,\zeta}=
\mathscr{B}_{k,m}^{n,\zeta}\setminus \mathscr{A}_{k+1,m}^{n,\zeta}$
has maximum possible dimension, so that only few
multiples of $P_{k+1}^{n,\zeta}$ need to be added
to span the space of polynomials of degree $\leq m$.
\end{remark}

\begin{remark}
Similar to Lemma~\ref{lemur}, the condition $\Psi(n)=m_{0}\geq 3\frac{n}{2}-1$
can be equivalently expressed by the vanishing of 
$2m_{0}-3n+3\geq 1$ sequences
of $(m_{0}+1)\times (m_{0}+1)$-subdeterminants of 
$(m_{0}+1)\times 3(m_{0}-n+1)$-matrices whose entries are the coefficients
of three successive best approximation vectors, arranged in 
certain circulant ways. Lemma~\ref{lemur} is recovered
with $m_{0}=3\frac{n}{2}-1$ and a single sequence. 
\end{remark}

Note that if all $P_{k-1}^{n,\zeta},P_{k}^{n,\zeta},P_{k+1}^{n,\zeta}$ have degree precisely $n$, then
$\mathscr{B}_{k,m}^{n,\zeta}$ becomes
\begin{equation} \label{eq:zingunga}
\mathscr{B}_{k,m}^{n,\zeta}=
\bigcup_{0\leq j\leq m-n} \left\{ T^{j}\cdot P_{k-1}^{n,\zeta}(T), \; T^{j}\cdot P_{k}^{n,\zeta}(T),\;  T^{j}\cdot P_{k+1}^{n,\zeta}(T)\right\}.
\end{equation}
This assumption holds in particular if $\zeta$ satisfies
\begin{equation} \label{eq:hoffnung}
\widehat{w}_{n}(\zeta)>w_{n-1}(\zeta),
\end{equation}
with irreducible polynomials 
$P_{k-1}^{n,\zeta},P_{k}^{n,\zeta},P_{k+1}^{n,\zeta}$.
For $n\geq 2$ and generic $\zeta$, we expect 
$\Psi(n,\zeta)=\widetilde{\Psi}(n,\zeta)=\lceil \frac{3n}{2}-1\rceil$. The crucial question is by how much the quantities 
$\Psi(n,\zeta), \widetilde{\Psi}(n,\zeta)$ can exceed
this value for certain biased $\zeta$. Some rather
easy estimates for $\Psi(n)$ and $\widetilde{\Psi}(n)$ are
summarized below.

\begin{lemma} \label{lehmer}
If $(n,\zeta)$ has the property $span(m_{0})$ 
and $\widetilde{span}(m_{0})$ respectively for some
$m_{0}$, then it has
the respective property for all $m\geq m_{0}$. 
The quantities $\Psi(n,\zeta), \widetilde{\Psi}(n,\zeta),
\Psi(n), \widetilde{\Psi}(n)$ are all well-defined and
for any $n\geq 2$ we have
\begin{equation} \label{eq:upbo}
\Psi(n,\zeta) \leq \widetilde{\Psi}(n,\zeta)\leq 2n-1,
\qquad\qquad \frac{3n}{2}-1\leq\Psi(n)\leq \widetilde{\Psi}(n)\leq 2n-1.
\end{equation}
Moreover, we have $\Phi(2)=2$.
Furthermore, if 
\eqref{eq:hoffnung} is satisfied, then 
$\Psi(n,\zeta)=\widetilde{\Psi}(n,\zeta)$.
\end{lemma}

The proof of Lemma~\ref{lehmer}
will be carried out in Section~\ref{proofs}. 
The problem to determine, or at least find good upper bounds, 
for $\Psi(n)$ and $\widetilde{\Psi}(n)$, is at the very 
core of our strategy
to find upper bounds for $\widehat{w}_{n}(\zeta)$.
Our result reads as follows. 

\begin{theorem} \label{cumbersome}
Let $n\geq 2$ be an integer and $\zeta$ be a transcendental
real number and assume \eqref{eq:burt} holds.
Then
for $\widetilde{\Psi}=\widetilde{\Psi}(n,\zeta)$ we have 
\begin{equation} \label{eq:jor}
\widehat{w}_{n}(\zeta)\leq \mathcal{D}_{n}(\widetilde{\Psi}):=\frac{2\widetilde{\Psi}-n+1+\sqrt{4\widetilde{\Psi}^{2}+17n^{2}-16\widetilde{\Psi} n+8\widetilde{\Psi}-18n+5}}{2}.
\end{equation}
Similarly, for $\Psi=\Psi(n,\zeta)$ we have 
\begin{equation} \label{eq:fehlt}
\widehat{w}_{n}(\zeta)\leq \mathcal{E}_{n}(\Psi):=
\frac{\Psi+1-\sqrt{\Psi^{2}-4\Psi n+8n^{2}+2\Psi-12n+5}}{2}.
\end{equation}

%
%
%
\end{theorem}


As for Theorem~\ref{gutsatz},
there are again several alternative conditions to
\eqref{eq:burt} with the same implications \eqref{eq:jor}
and \eqref{eq:fehlt}, and we believe none of them is in fact required.
We point out that upon the condition \eqref{eq:hoffnung}
we have the stronger bound \eqref{eq:jor} anyway, in view
of Lemma~\ref{lehmer}. 
Both $\mathcal{D}_{n}(t), \mathcal{E}_{n}(t)$ increase
as functions of $t$ and
$\mathcal{D}_{n}(t)\leq \mathcal{E}_{n}(t)$ 
for $t\geq 3\frac{n}{2}-1$.
For $n=2$ from \eqref{eq:fehlt} and $\Phi(2)=2$ we
get yet another proof of \eqref{eq:najo}. 
The identities $\mathcal{D}_{n}(2n-1)=2n-1$ and
$\mathcal{E}_{n}(2n-1)=2n-1$
both confirm
$\widehat{w}_{n}(\zeta)\leq 2n-1$ from~\cite{davsh} again.
When $n$ is even and the generic case 
$\widetilde{\Psi}=\widetilde{\Psi}(n,\zeta)=3\frac{n}{2}-1$ occurs,
we have $\mathcal{D}_{n}(\widetilde{\Psi})=\sigma_{n}$ by construction.

We highlight the case $\widetilde{\Psi}(n,\zeta)=2n-d$ and
$\Psi(n,\zeta)=2n-d$ with fixed $d>0$, 
as $n\to\infty$.
For $d=2$, that is 
$\widetilde{\Psi}(n,\zeta)=2n-2$, the estimate \eqref{eq:jor}
naturally recovers the bound $\theta_{n}$ from \eqref{eq:neu}.
In general the bounds are of order
\begin{equation} \label{eq:ende}
\mathcal{D}_{n}(2n-d)= 2n-d+o(1),
\qquad \mathcal{E}_{n}(2n-d)= 2n-\frac{d+1}{2}+o(1),
\end{equation}
as $n\to\infty$, with positive 
error terms.
We compare Theorem~\ref{cumbersome} with 
the bounds 
\eqref{eq:frwe} and \eqref{eq:asyb} from Section~\ref{intro},
where it is useful to keep the notation $\widetilde{\Psi}(n,\zeta)=2n-d$ and
$\Psi(n,\zeta)=2n-d$ with fixed $d>0$.
It turns out that \eqref{eq:jor} is stronger as soon as
$d\geq 3$, whereas \eqref{eq:fehlt} is stronger for
$d\geq 4$, upon $m=2n-d\geq 3\frac{n}{2}-1$
or equivalently $n\geq 2d-2$.
We finally remark that estimates analogous to \eqref{eq:derive}
in terms of $m,n$ can be stated.

\section{Preliminaries}

\subsection{Successive minima and
parametric geometry of numbers} \label{schmidt}

We introduce the concept of parametric geometry of numbers following Schmidt and Summerer~\cite{ss},
where we slightly deviate in the notation and
only treat the linear form case relevant to us.
We prefer to state the results of this section for polynomials of degree at most $m$
instead of $n$ to avoid confusion later. Indeed in the proofs 
of Theorems~\ref{thm},~\ref{gutsatz},~\ref{cumbersome}, 
although their claims concern approximation in dimension
$n$, we will transition to an approximation problem 
in certain dimensions $m>n$, corresponding to the present $m$.
For given $m\geq 1$, a real number $\zeta$
and every $1\leq j\leq m+1$, define the 
functions $\psi_{m,j}^{\ast}(Q)$ parametrized by 
$Q$ as the maximum value $\eta$ such that
\begin{equation} \label{eq:modell}
H(P)\leq Q^{\frac{1}{m}+\eta}, \qquad \vert P(\zeta)\vert \leq Q^{-1-\eta}
\end{equation}
has $j$ linearly independent solutions in integer polynomials $P$ of degree at most $m$. 
Define further 
\[
\underline{\psi}_{m,j}^{\ast}= \liminf_{Q\to\infty} \psi_{m,j}^{\ast}(Q), \qquad
\overline{\psi}_{m,j}^{\ast}= \limsup_{Q\to\infty} \psi_{m,j}^{\ast}(Q).
\]
The system \eqref{eq:modell} can be equivalently formulated as a problem of determining the successive minima
of a convex body (parametrized by $Q$) with respect to a fixed lattice, both in $\mathbb{R}^{m+1}$, 
see~\cite{ss} for details. For convenience the derived functions
\begin{equation} \label{eq:logar}
L_{m,j}^{\ast}(q):= q \psi_{m,j}^{\ast}(q), \qquad q=\log Q,
\end{equation}
were introduced in~\cite{ss}. 

\begin{definition} \label{schmus}
For given $m\geq 1$ and $\zeta$ a real number, we call the
image of the functions
\[
(L_{m,1}^{\ast}(q), L_{m,2}^{\ast}(q), \ldots,
L_{m,m+1}^{\ast}(q)), \qquad\qquad q>0,
\]
in the plane defined as above the {\em combined Schmidt-Summerer graph 
associated to $(m,\zeta)$}.
\end{definition}

We gather some properties of the combined graph.
Any $L_{m,j}^{\ast}(q)$ coincides locally with some function
\begin{equation} \label{eq:piscibus}
L_{P}^{\ast}(q)= \max \left\{ \log H(P)-\frac{q}{m}, \log \vert P(\zeta)\vert+q \right\}
\end{equation}
related to a polynomial $P\in\mathbb{Z}[T]$ of degree at most $m$.
Hence the functions $L_{m,j}^{\ast}(q)$ are piecewise linear and have slope among $\{-1/m,1\}$. Conversely, any best approximation polynomial $P_{k}$ associated to $m,\zeta$ 
as in Definition~\ref{newdef},
induces the function $L_{m,1}^{\ast}(q)$ in some non-empty
interval $I_{k}$, that is
\[
L_{P_{k}}^{\ast}(q)=
\max \left\{ \log H(P_{k})-\frac{q}{m}, \log \vert P_{k}(\zeta)\vert+q \right\}=L_{m,1}^{\ast}(q), \qquad k\geq 1,\; q\in I_{k}.
\]

Minkowski's second Convex Body Theorem~\cite{mink} implies the estimation
\begin{equation} \label{eq:lsumme}
\left\vert \sum_{j=1}^{m+1} L_{m,j}^{\ast}(q)\right\vert \leq C(m)
\end{equation}
for an absolute constant $C(m)>0$,
uniformly in the parameter $q$. As a consequence
\begin{equation} \label{eq:nurmi}
L_{m,m+1}^{\ast}(q) \geq -\sum_{j=1}^{m} L_{m,j}^{\ast}(q)-C(k) \geq  -mL_{m,m}^{\ast}(q)-C(m).
\end{equation}
More generally this argument yields
\begin{equation} \label{eq:gimme}
L_{m,m+1}^{\ast}(q) \geq \left(-\frac{j}{m+1-j}\right) \cdot L_{m,j}^{\ast}(q)-O(1), \qquad 1\leq j\leq m,
\end{equation}
see also the introduction in~\cite{ssch}. 
It will be convenient to use the following parametric 
estimates connecting classical exponents with $L_{m,j}^{\ast}$.

\begin{lemma} \label{apply}
Let $m\geq 1$ be an integer, $\zeta$ a real number
and $j\in\{ 1,2,\ldots,m+1\}$. For linearly 
independent integer polynomials $P_{1},\ldots,P_{j}$ of
degree at most $m$, define $H,w$ by
\[
\max_{1\leq i\leq j} H(P_{i})= H, \qquad
\max_{1\leq i\leq j} \vert P_{i}(\zeta)\vert= H^{-w}.
\]
Without loss of generality assume a labeling such that 
$\vert P_{1}(\zeta)\vert= H^{-w}$.
Then for $q= \tilde{q}:= \frac{m}{m+1}(\log H-w)$ the solution to 
$\log H-q/m= w+q$ we have
\begin{equation} \label{eq:umrechnen}
L_{m,j}^{\ast}(\tilde{q})\leq L_{P_{1}}^{\ast}(\tilde{q})
= \tilde{q}\left(\frac{m+1}{m}\cdot \frac{1}{1+w}-\frac{1}{m}\right)=
\tilde{q}\cdot \frac{m-w}{m(1+w)}. 
\end{equation}
\end{lemma}

\begin{proof}
Define the auxiliary function
\[
\beta (q):= \max\{ \log H-\frac{q}{m}, w+q\}, \qquad \qquad q>0.
\]
We first claim that all graphs of the functions
$L_{P_{i}}^{\ast}(q)$ lie below the graph of $\beta$,
that is $\beta(q)\geq L_{P_{i}}^{\ast}(q)$ for 
$1\leq i\leq j$ and $q>0$.
By $H(P_{i})\leq H$ we have the inequality for $q=0$.
Then $\beta$ and all $L_{P_{i}}^{\ast}(q)$ decay with slope $-1/m$,
so the claim is true for any 
$L_{P_{i}}^{\ast}$ and $q\geq q_{(i)}$ where $q_{(i)}$
is the point where $L_{P_{i}}^{\ast}$ starts to rise with slope $1$, that is where equality of the expressions in \eqref{eq:piscibus} holds
for $P=P_{i}$.
By the property $\vert P_{i}(\zeta)\vert\leq H^{-w}$
we know that the rising phase $-\log \vert P_{i}(\zeta)\vert+q$
lies entirely below the one of $\beta$ which is $w+q$,
with equality for $i=1$ and $q$ large enough that $\beta$
and $L_{P_{1}}^{\ast}$ both rise. So the claim
is true for $q\geq q_{0}:= \max q_{(i)}$. If we had
$\beta(q_{1})<L_{P_{i}}^{\ast}(q_{1})$
for some $i$ and some $q_{1}>0$, then by the above findings 
at $q_{1}$ the function $L_{P_{i}}^{\ast}$ must already rise
with slope $1$. However, since $\beta$ has slope at most $1$,
we would have the inequality $\beta(q)<L_{P_{i}}^{\ast}(q)$
for all $q\geq q_{1}$. This
contradicts our second finding and proves the claim.
It is easily verified that $\tilde{q}$
is the point where 
$\beta$ changes slope from $-1/m$ to $1$, and that
$\beta(\tilde{q})$ equals the right hand side in \eqref{eq:umrechnen}. It is further readily verified that
$L_{P_{1}}^{\ast}$ rises at $\tilde{q}$, and we infer
the identities in \eqref{eq:umrechnen} from the identity
statement above. Finally, since the $P_{i}, 1\leq i\leq j$, are linearly independent by assumption, we conclude
$L_{m,j}^{\ast}(\tilde{q})\leq \max_{1\leq i\leq j} L_{P_{i}}^{\ast}(\tilde{q})=
L_{P_{1}}^{\ast}(\tilde{q})=\beta(\tilde{q})$. 
The proof of \eqref{eq:umrechnen} is thus complete.
\end{proof}

We remark that Lemma~\ref{apply} and its proof are closely 
connected to the identities
\begin{equation} \label{eq:elaine}
(w_{m}(\zeta)+1)\left(\frac{1}{m}+\underline{\psi}_{m,1}^{\ast}\right)=
(\widehat{w}_{m}(\zeta)+1)\left(\frac{1}{m}+\overline{\psi}_{m,1}^{\ast}\right)=
\frac{m+1}{m},
\end{equation}
established in~\cite[Theorem~1.4]{ss}. A generalization for
higher successive minima as in Lemma~\ref{apply} was 
pointed out in~\cite{j2}.


\subsection{Certain estimates for polynomials
and Diophantine exponents} \label{prepa}



%
%

In this section
we provide several, mostly elementary, facts on the
relation between best
approximation polynomials from Definition~\ref{newdef} and the exponents $w_{n}(\zeta),
\widehat{w}_{n}(\zeta)$. We will be concerned with irreducibilty
of best approximation polynomials. First we recall 
that as pointed out in~\cite{wirsing}, for 
any given $n,\zeta$ and any $w<w_{n}(\zeta)$ there exist 
infinitely many
{\em irreducible} integer polynomials of degree at most
$n$ for which \eqref{eq:w} holds. This follows from a 
general estimate
on the height of products of polynomials, sometimes referred to as Gelfond's Lemma. It states that
for any positive integer $n$ there exists an absolute
constant $K(n)$ such that 
\begin{equation} \label{eq:wirsing}
K(n)^{-1}H(P)H(Q) \leq H(PQ) \leq K(n)H(P)H(Q)
\end{equation}
holds for all polynomials $P,Q\in\mathbb{Z}[T]$ of 
degree at most $n$. This estimate and the corollary above
on irreducibility were also used by Wirsing in~\cite{wirsing}. 
The same kind of argument, already used in~\cite{buschl},
yields that in case of 
$w_{n}(\zeta)>w_{n-1}(\zeta)$, 
infinitely many among the best approximation polynomials 
associated to $(n,\zeta)$ are 
irreducible of degree {\em precisely} $n$.
Similarly, the stronger claim 
$\widehat{w}_{n}(\zeta)>w_{n-1}(\zeta)$
implies the same property for all best approximation polynomials
of sufficiently large index/height, as already
observed in the proof of~\cite[Lemma~3.2]{ichindag}. 
For simplicity we put $P_{k}=P_{k}^{n,\zeta}$ for the next result.

\begin{proposition} \label{pr}
Let $n\geq 1$ be an integer, $\zeta$ be a transcendental
real number and
$(P_{k})_{k\geq 1}$ the sequence of best approximation
polynomials associated to $(n,\zeta)$.
For any
$\epsilon>0$ and all sufficiently large $k$, upon
denoting $X_{k+1}=H(P_{k+1})$ we have
\begin{equation} \label{eq:by}
H(P_{k})<H(P_{k+1})= X_{k+1}, \qquad
\vert P_{k+1}(\zeta)\vert<\vert P_{k}(\zeta) \vert 
\leq X_{k+1}^{-\widehat{w}_{n}(\zeta)+\epsilon}.
\end{equation}
\end{proposition}

\begin{proof}
We only have to show $\vert P_{k}(\zeta) \vert 
\leq X_{k+1}^{-\widehat{w}_{n}(\zeta)+\epsilon}$.
Suppose otherwise
$\vert P_{k}(\zeta)\vert>
X_{k+1}^{-\widehat{w}_{n}(\zeta)+\epsilon}$.
Since $P_{k+1}$ is the best approximation polynomial
succeeding $P_{k}$, for any non-zero $Q$ of degree at most $n$ and height 
strictly less than $X_{k+1}$ we have
$\vert Q(\zeta)\vert\geq 
\vert P_{k+1}(\zeta)\vert>X_{k+1}^{-\widehat{w}_{n}(\zeta)+\epsilon}$. 
This obviously contradicts the 
definition of $\widehat{w}(\zeta)$ for large $k$.
\end{proof}

Proposition~\ref{pr} tells us that
if $\widehat{w}_{n}(\zeta)$ is large, infinitely
often there are two linearly independent best approximation
polynomials with small evaluation at $\zeta$. The claim
is closely related the more general to~\cite[Theorem~1.1]{ss}
and the derived inequalities
$\underline{\psi}_{n,j+1}^{\ast}\leq \overline{\psi}_{n,j}^{\ast}$
for $1\leq j\leq n$ for $\mathbb{Q}$-linearly independent
real vectors $(1,\zeta_{1},\ldots,\zeta_{n})$, taking $j=1$.

\begin{lemma}  \label{ehklar}
Let $n\geq 1$ be an integer, $\zeta$ be a transcendental real number 
and $(P_{k})_{k\geq 1}$ be the sequence
of best approximation polynomials associated to $(n,\zeta)$. 
Let $\varepsilon>0$. Then for all large $k\geq k_{0}(\varepsilon)$
we have
\[
H(P_{k+1}) \leq H(P_{k})^{w_{n}(\zeta)/\widehat{w}_{n}(\zeta)+\varepsilon}.
\]
\end{lemma}

\begin{proof}
By definition of $w_{n}(\zeta)$ for any $\epsilon>0$ and 
large $k\geq k_{0}(\epsilon)$ we have
\[
\vert P_{k}(\zeta)\vert \geq H(P_{k})^{-w_{n}(\zeta)-\epsilon}.
\]
Let
\[
X_{k}=H(P_{k})^{(w_{n}(\zeta)+\epsilon)/(\widehat{w}_{n}(\zeta)-\epsilon)}.
\]
Then 
\[
\vert P_{k}(\zeta)\vert \geq X_{k}^{-\widehat{w}_{n}(\zeta)+\epsilon}.
\]
On the other hand, 
by definition of $\widehat{w}_{n}(\zeta)$ for large $k$ there has to be polynomial
$Q_{k}$ of height less then $X_{k}$ for which 
$\vert Q_{k}(\zeta)\vert<X_{k}^{-\widehat{w}_{n}(\zeta)+\epsilon}\leq \vert P_{k}(\zeta)\vert$.
By definition of the best approximation polynomials we see that $Q_{k}=P_{k+1}$ is a suitable choice.  
Hence
\[
H(P_{k+1})\leq X_{k}= H(P_{k})^{(w_{n}(\zeta)+\epsilon)/(\widehat{w}_{n}(\zeta)-\epsilon)}.
\]
The claim follows
since $\epsilon$ can be chosen arbitrarily small, 
for $\varepsilon$ some modification of $\epsilon$.
\end{proof}

The next result will be applied
to best approximation polynomials as well 
in order to simplify the proof of our main results,
and establish the upper bounds
in \eqref{eq:upbo} of Lemma~\ref{lehmer}.

\begin{lemma} \label{lainyweg}
Let $P(T),Q(T)$ be two integer polynomials of degrees $a,b$
respectively, without common factor. Then the set of polynomials
\[
\Omega:=\left\{ P, TP,\ldots, T^{b-1} P, 
Q, TQ,\ldots, T^{a-1}Q\right\},
\]
is linearly independent, and thus spans the space of
polynomials of degree at most $a+b-1$ in a direct sum.
\end{lemma}

\begin{proof}
Assume the claim is false and
$\Omega$ is linearly dependent. Then by the structure 
of $\Omega$, there exist polynomials $A,B$ not both identically zero and of degree at most $n-1<n$ such that
$AP\equiv BQ$. However, since $P$ and $Q$ have
no common factor and $B$ has degree less than $P$, we cannot have such an identity
by the unique factorization in $\mathbb{Z}[T]$.
\end{proof}

Essentially the same claim was already implicitly 
used within
the proof of~\cite[Theorem~2.1]{ichindag}, where 
an equivalent proof using the resultant was given.
For the immediate concern of
Theorem~\ref{thm} we state a direct corollary with $a=b=n$.

\begin{corollary} \label{nocom}
Let $P,Q$ be two integer polynomials both of degree precisely
$n$ and without common factor. Then the set of polynomials
\[
\mathscr{P}:=\bigcup_{0\leq j\leq n-2}
\left\{ T^{j}P, \; T^{j}Q\right\},
\]
is linearly independent and thus spans a $2n-2$
dimensional hyperspace of the vector space of polynomials
of degree at most $2n-2$.
\end{corollary}

Finally we recall two facts from~\cite{buschl}.
A special case of~\cite[Theorem~2.2]{buschl}
shows that the condition \eqref{eq:burt}, that is 
$w_{n}(\zeta)>w_{n-1}(\zeta)$, implies
\begin{equation} \label{eq:puschel}
\widehat{w}_{n}(\zeta) \leq n+(n-1)\frac{\widehat{w}_{n}(\zeta)}{w_{n}(\zeta)}.
\end{equation}  
It is hard to tell if \eqref{eq:puschel}
holds without the imposed condition. See~\cite{ichcomments}
for estimates in the general case.
Moreover we will need the estimate
\begin{equation} \label{eq:mundn}
\min\{ w_{n_{1}}(\zeta),\widehat{w}_{n_{2}}(\zeta)\}\leq n_{1}+n_{2}-1, 
\end{equation}
valid for any positive integers $n_{1},n_{2}$ and any
transcendental real number $\zeta$, as shown 
in~\cite[Theorem~2.3]{buschl}. Notice the choice $n_{1}=n_{2}=n$ 
recovers $\widehat{w}_{n}(\zeta)\leq 2n-1$ 
due to Davenport and Schmidt.

\section{Proofs} \label{proofs}

We start with the proofs of the lemmata in Section~\ref{n}.

\begin{proof}[Proof of Lemma~\ref{lemur}]
Write the polynomials $A_{k}, B_{k}, C_{k}$ 
in coordinates as well, such that glued together
they form vectors 
\[
\underline{v}_{k}=(a_{k,0},\ldots,a_{k,n/2-1},b_{k,0},\ldots,
b_{k,n/2-1},c_{k,0},\ldots,c_{k,n/2-1})\in\mathbb{Z}^{3n/2}.
\]
Multiplying out the product in
\eqref{eq:bedingung} using \eqref{eq:dingung}
and assuming all coefficients to vanish
yields a system of $3\frac{n}{2}$ linear equations in
$3\frac{n}{2}$ variables. It can
be checked that this system corresponds
to $\Lambda_{n}(\underline{h}_{k}^{n,\zeta})
\cdot \underline{v}_{k}^{T}=\underline{0}$, with
the matrix $\Lambda_{n}(\underline{h}_{k}^{n,\zeta})$
from Section~\ref{n} and where $\underline{x}^{T}$ denotes the transpose
of $\underline{x}$. 
From linear algebra
this system has a non-trivial solution $\underline{v}_{k}$ 
if and only if the corresponding 
matrix $\Lambda_{n}(\underline{h}_{k}^{n,\zeta})$ is singular, or
equivalently $\Phi_{n}(\underline{h}_{k}^{n,\zeta})=0$.
Hence the first two conditions are equivalent. 
Since any polynomial $P_{k-1}^{n,\zeta}A_{k}, P_{k}^{n,\zeta}B_{k}, P_{k+1}^{n,\zeta}C_{k}$ obviously
lies in the span of
the polynomials in \eqref{eq:ingung}, and conversely
the latter form any possible linear combination as in \eqref{eq:bedingung},
the last two assertions are equivalent as well.
\end{proof}

\begin{proof}[Proof of Lemma~\ref{lehmer}]
The claim $\Phi(2)=2$ resembles the remark
towards the end of Section~\ref{n} on three consecutive best approximation polynomials
being linearly independent infinitely often. 
Concerning the upper bound in \eqref{eq:upbo},
one can even restrict to
two successive best approximation polynomials. 
As pointed out in Remark~\ref{hirsch} infinitely often
two successive best approximation polynomials 
$P_{k-1}=P_{k-1}^{n,\zeta}, P_{k}=P_{k}^{n,\zeta}$ have 
no common factor~\cite[Theorem~2b]{davsh}. 
For such $P_{k-1}, P_{k}$ of degrees
$d_{k-1}, d_{k}$ respectively,
first observe
that the set
\begin{equation} \label{eq:lemmer}
\Omega_{k}:=\left\{ P_{k-1}, TP_{k-1},\ldots, T^{d_{k}-1}P_{k-1},
 P_{k}, TP_{k},\ldots, T^{d_{k-1}-1}P_{k}\right\},
\end{equation}
is linearly independent by Lemma~\ref{lainyweg}.
Observe every element in $\Omega_{k}$
has degree at most $d_{k-1}+d_{k}-1$.
Now when we add the polynomials in
$\mathscr{Q}_{k}:=\{ T^{d_{k}}P_{k-1},\ldots,
T^{m-d_{k-1}}P_{k-1}\}$ one by one to $\Omega_{k}$, 
the dimension of the span must
increase in every step since the new polynomial has higher degree than
any polynomial in the old span. 
We see that $\Omega_{k}\cup \mathscr{Q}_{k}$ consists
of linearly independent polynomials of
degree at most $m$ and
has cardinality 
$d_{k}+d_{k-1}+m-(d_{k}+d_{k-1}-1)=m+1$. Hence
indeed $\Psi(n)\leq 2n-1$ and the quantities
$\widetilde{\Psi}, \Psi$ 
are well-defined. A very similar inductive argument ensures the
first claim of the lemma. We next prove
the lower bound in \eqref{eq:upbo}. Choose $\zeta$ that satisfies
\eqref{eq:hoffnung}. This is possible since the set of $\zeta$
with
$w_{n-1}(\zeta)\geq n$ has Hausdorff dimension
$n/(n+1)<1$ by Baker and Schmidt~\cite{baks}
and Bernik~\cite{bernik}, whereas $\widehat{w}_{n}(\zeta)\geq n$, i.e.
\eqref{eq:wmono}, holds for
any transcendental real $\zeta$. The condition \eqref{eq:hoffnung}
ensures $d_{k-1}=d_{k}=d_{k+1}=n$ for $d_{j}$ the degree
of $P_{j}^{n,\zeta}$. The estimate 
$\Psi(n,\zeta)\geq 3\frac{n}{2}-1$ for such $\zeta$
follows as we have only
$3(m-n+1)<m+1$ polynomials in the union in \eqref{eq:zingunga}
if $m<3\frac{n}{2}-1$. Thus in particular
$\Psi(n)\geq 3\frac{n}{2}-1$.
The last claim follows similarly from $d_{k-1}=d_{k}=d_{k+1}=n$
upon \eqref{eq:hoffnung},
and the linear independence of $\Omega_{k}$.
%
%
\end{proof}

We remark that the quoted result~\cite[Theorem~2b]{davsh} was
a crucial observation to infer the bound $2n-1$ 
for $\widehat{w}_{n}(\zeta)$ in that paper.
Now we turn towards the proofs of the main results Theorems~\ref{thm},
~\ref{gutsatz},~\ref{cumbersome}. 
The key idea of all proofs
is to blow up the dimension of the problem from $n$
to some $m$ and observe that the assumption of 
large $\widehat{w}_{n}(\zeta)$ conflicts with Minkowski's Second Convex Body Theorem
in this modified approximation problem, related to the combined Schmidt-Summerer
graph in dimension $m$ from Definition~\ref{schmus}. This contradiction
will essentially yield the respective
upper bounds. For Theorem~\ref{thm}, concretely 
we choose $m=2n-2$. 
The method already requires some subtle application of results from 
Section~\ref{prepa} to work properly. The technical problems
increase as $m<2n-2$ decreases, 
thus leading only to conditional bounds so far, see below. 

\begin{proof}[Proof of Theorem~\ref{thm}]
We proof the claim indirectly.
Assume the claim of the theorem is false, 
that is there exists an integer $n$ and a
real number $\zeta$ such that 
$\widehat{w}_{n}(\zeta)>\theta_{n}$. We may assume 
$\zeta$ is transcendental in view of identity \eqref{eq:alg} for
an algebraic number $\zeta$. 
Since $\theta_{n}>2n-2$ we conclude $\widehat{w}_{n}(\zeta)>2n-2$. 
Hence application of the relation \eqref{eq:mundn} with 
$n_{1}=n, n_{2}=n-1$ yields 
\begin{equation} \label{eq:eqq}
w_{n-1}(\zeta)\leq 2n-2<\theta_{n}<\widehat{w}_{n}(\zeta).
\end{equation}
For simplicity write $P_{k}=P_{k}^{n,\zeta}$ for
$(P_{k}^{n,\zeta})_{k\geq 1}$ be the sequence of best approximation
polynomials associated to $(n,\zeta)$ as in Definition~\ref{newdef}. 
As carried out in Section~\ref{prepa}, Wirsing's estimate \eqref{eq:wirsing} and \eqref{eq:eqq}
imply that for any large $k$ the best 
approximation polynomial $P_{k}$
is irreducible of degree precisely $n$. 
Now consider the polynomial approximation problem for polynomials of degree at most $2n-2$ in the variable $\zeta$, related
to the combined Schmidt-Summerer graph 
associated to $(2n-2,\zeta)$.
For any $k\geq 1$, let 
\[
V_{k,j}(T)= T^{j}P_{k}(T), \qquad 0\leq j\leq n-2,
\]
and further define the sets 
\[
\mathscr{V}_{k}= \bigcup_{j=0}^{n-2} V_{k,j}, \qquad \mathscr{P}_{k}= \mathscr{V}_{k-1}\cup \mathscr{V}_{k}.
\]
We see that $\mathscr{P}_{k}$ consists of $2n-2$ polynomials with integer coefficients and degree at most $2n-2$.
Since $P_{k-1}, P_{k}$
are distinct and irreducible of degree precisely $n$, we may apply
Corollary~\ref{nocom} with $P=P_{k-1}, Q=P_{k}$ and
$\mathscr{P}=\mathscr{P}_{k}$
and see that for any sufficiently large $k$, 
the set $\mathscr{P}_{k}$
is linearly independent and forms
a basis of a $2n-2$ dimensional hyperspace of the space of polynomials of degree at most $2n-2$. 
Additionally consider the subsequent best approximation polynomial $P_{k+1}$ and the set
\[
\mathscr{R}_{k}= \mathscr{P}_{k}\cup \mathscr{V}_{k+1}=\mathscr{V}_{k-1}\cup \mathscr{V}_{k}\cup \mathscr{V}_{k+1}.
\]
We identify a polynomial with its coefficient vector 
in $\mathbb{Z}^{2n-1}$ in the sequel. We distinguish two cases.

Case 1: There exist arbitrarily large $k$ such that $\mathscr{R}_{k}$ spans the space
of polynomials of degree at most $2n-2$. In other words for infinitely many $k$ there exists some polynomial
$R_{k+1}$ in $\mathscr{V}_{k+1}$ not included in the span of $\mathscr{P}_{k}$.
Let $\epsilon\in(0,\theta_{n}-(2n-2))$ be arbitrary but fixed. 
Observe that 
\begin{equation} \label{eq:noned}
H(P_{k-1})<
H(P_{k})=H(V_{k,0})=H(V_{k,1})=\cdots=H(V_{k,n-2}), \qquad k\geq 1.
\end{equation}
Moreover without loss of generality we may 
assume that $\zeta\in(0,1)$ and hence
\begin{equation}  \label{eq:haha} 
\vert P_{k}(\zeta)\vert = \vert V_{k,0}(\zeta)\vert = \max_{0\leq j\leq n-2} \vert V_{k,j}(\zeta)\vert, \qquad k\geq 1.
\end{equation}
For large $k$, by \eqref{eq:by} from 
Proposition~\ref{pr} and \eqref{eq:haha}, we have
\begin{equation} \label{eq:glei}
\vert P(\zeta)\vert \leq H(P_{k})^{-\widehat{w}_{n}(\zeta)+\epsilon}, \qquad P\in \mathscr{P}_{k}.
\end{equation}
We consider the combined Schmidt-Summerer
graph associated to $(m,\zeta)$ with $m=2n-2$.
as defined in Section~\ref{schmidt}. 
Keep in mind that \eqref{eq:noned} and \eqref{eq:haha} with \eqref{eq:piscibus} imply
\begin{equation} \label{eq:nuned}
\max_{P\in\mathscr{V}_{k}} L_{P}^{\ast}(q) = L_{P_{k}}^{\ast}(q), \qquad \qquad k\geq 1,\; q>0.
\end{equation}
First look at the points $(q_{k},L_{P_{k}}^{\ast}(q_{k}))$
where the graphs of $L_{P_{k-1}}^{\ast}$ and $L_{P_{k}}^{\ast}$ intersect, that is
$q_{k}$ is defined via
\begin{equation} \label{eq:jgleich}
L_{P_{k-1}}^{\ast}(q_{k})= L_{P_{k}}^{\ast}(q_{k}).
\end{equation}

Since obviously at such points $L_{P_{k-1}}^{\ast}$ rises
and $L_{P_{k}}^{\ast}$ decays,
by \eqref{eq:piscibus} the value $q_{k}$
is implicitly defined by
\[
\log H(P_{k})-\frac{q_{k}}{2n-2}= \log \vert P_{k-1}(\zeta)\vert + q_{k},
\]
however, we will not need this directly. In view of \eqref{eq:nuned} we know that 
\begin{equation} \label{eq:hehe}
L_{P}^{\ast}(q_{k})\leq L_{P_{k-1}}^{\ast}(q_{k})=L_{P_{k}}^{\ast}(q_{k}), \qquad P\in\mathscr{P}_{k},
\end{equation}
on the other hand the 
differences $\vert L_{P}^{\ast}(q_{k})-L_{Q}^{\ast}(q_{k})\vert$ 
among $P,Q\in\mathscr{P}_{k}$
is uniformly bounded for $k\geq 1$.

From \eqref{eq:jgleich} and \eqref{eq:hehe} and since the span of
$\mathscr{P}_{k}$ has full dimension $2n-2$, we infer $L_{2n-2,2n-2}^{\ast}(q_{k}) \leq L_{P_{k}}^{\ast}(q_{k})$.
In fact, in view of
\eqref{eq:noned} and \eqref{eq:glei},
we may apply Lemma~\ref{apply} to
$j=2n-2$, the polynomials in $\mathscr{P}_{k}$
and a parameter $w\geq \widehat{w}_{n}(\zeta)-\epsilon$, 
with $q=q_{k}$. 
From its claim \eqref{eq:umrechnen} we obtain
\begin{equation} \label{eq:jetzaber}
L_{2n-2,2n-2}^{\ast}(q_{k}) \leq L_{P_{k}}^{\ast}(q_{k})\leq 
q_{k}\cdot \frac{2n-2-\widehat{w}_{n}(\zeta)}{(2n-2)(1+\widehat{w}_{n}(\zeta))}+\tilde{\epsilon}q_{k}
\end{equation}
where $\tilde{\epsilon}$ is some variation of $\epsilon$ which tends to $0$ as $\epsilon$ does.
Observe the expression will be negative
when $\epsilon$ (or $\tilde{\epsilon}$) is small 
enough since $\widehat{w}_{n}(\zeta)>\theta_{n}>2n-2$ by assumption.
Hence, with \eqref{eq:nurmi}, for the last successive
minimum function we conclude
\begin{equation}  \label{eq:hand}
L_{2n-2,2n-1}^{\ast}(q_{k}) \geq 
(2-2n)q_{k}\cdot\frac{2n-2-\widehat{w}_{n}(\zeta)}{(2n-2)(1+\widehat{w}_{n}(\zeta))}-(2n-2)\tilde{\epsilon}q_{k}+O(1).
\end{equation}
We now want to derive an upper bound 
for $L_{2n-2,2n-1}^{\ast}(q_{k})$ which is smaller under
our assumption $\widehat{w}_{n}(\zeta)>\theta_{n}$, which
will lead to the desired contradiction. By assumption of case 1 for infinitely many $k$
there exists $R_{k+1}\in\mathscr{V}_{k+1}=\mathscr{R}_{k} \setminus \mathscr{P}_{k}$
which does not lie in the span of $\mathscr{P}_{k}$.
Obviously $H(R_{k+1})=H(P_{k+1})$ by construction. 
To shorten the notation let $H_{l}=H(P_{l})$ for any 
integer $l\geq 1$. Then Lemma~\ref{ehklar} implies 
$H_{k+1}\leq H_{k}^{w_{n}(\zeta)/\widehat{w}_{n}(\zeta)+\epsilon}$. 
On the other hand, inequality \eqref{eq:puschel},
which can be applied since its condition $w_{n}(\zeta)>w_{n-1}(\zeta)$ is satisfied
by virtue of \eqref{eq:eqq} and \eqref{eq:wmono}, 
can be reformulated to $w_{n}(\zeta)/\widehat{w}_{n}(\zeta)\leq (n-1)(\widehat{w}_{n}(\zeta)-n)^{-1}$. 
Thus for large $k$ we deduce
\begin{equation}  \label{eq:jojo}
H_{k+1}\leq H_{k}^{\nu+\epsilon}, \qquad \nu= \frac{n-1}{\widehat{w}_{n}(\zeta)-n}.
\end{equation}

Since $(q_{k},L_{P_{k}}^{\ast}(q_{k}))$ lies in the graph 
of $L_{P_{k}}^{\ast}$ which decays with slope $-1/(2n-2)$ in a neighborhood $U_{k}$
of $q_{k}$ (since at $q_{k}$ it meets $L_{P_{k-1}}^{\ast}$ by \eqref{eq:jgleich} and $H_{k}>H_{k-1}$), we have
\[
\log H_{k}-\frac{q_{k}}{2n-2}= L_{P_{k}}^{\ast}(q_{k}).
\]
Together with \eqref{eq:jetzaber} we derive
\begin{equation} \label{eq:jordan}
\log H_{k}\leq \frac{q_{k}}{2n-2}+
q_{k}\cdot\frac{2n-2-\widehat{w}_{n}(\zeta)}{(2n-2)(1+\widehat{w}_{n}(\zeta))}+\tilde{\epsilon}q_{k}.
\end{equation}

It is not hard to see that the function $L_{R_{k+1}}^{\ast}$ decays at $q_{k}$ as well. We carry this out.
We have $L_{R_{k+1}}^{\ast}(q_{k})>0$ since otherwise we get a contradiction
to \eqref{eq:lsumme} as all $L_{2n-2,j}^{\ast}(q_{k})<0$ for $1\leq j\leq 2n-1$ and the first $2(m-n+1)$
are negative by some fixed multiple of $q_{k}$ 
(by assumption $\{ R_{k+1}\}\cup \mathscr{P}_{k}$
spans $\mathbb{R}^{2n-1}$ and for $P\in\mathscr{P}_{k}$ we have shown $L_{P}^{\ast}(q_{k})<-cq_{k}$ for some $c>0$). 
On the other hand $R_{k+1}$ induces an approximation of 
quality $-\log \vert R_{k+1}(\zeta)\vert/ \log H(R_{k+1})>2n-2$, 
and by equating the two expressions in \eqref{eq:piscibus}
in our present dimension $m=2n-2$, this leads to 
$L_{R_{k+1}}^{\ast}(r_{k+1})<0$
at the local
minimum $r_{k+1}$ of the function $L_{R_{k+1}}^{\ast}$.
Thus indeed we deduce that $r_{k+1}>q_{k}$ and $L_{R_{k+1}}^{\ast}$ decays at $q_{k}$.
Hence we also have
\[
\log H_{k+1}-\frac{q_{k}}{2n-2}= L_{R_{k+1}}^{\ast}(q_{k}).
\]
Clearly $L_{R_{k+1}}^{\ast}(q_{k})$ is the maximum among  $L_{S}^{\ast}(q_{k})$ for 
$S\in \{ R_{k+1}\}\cup \mathscr{P}_{k}$ since it is the only positive value.
With \eqref{eq:jojo} and since $\{R_{k+1},\mathscr{P}_{k}\}$ span $\mathbb{R}^{2n-1}$ we infer
\[
 L_{2n-2,2n-1}^{\ast}(q_{k}) \leq L_{R_{k+1}}^{\ast}(q_{k}) \leq (\nu+\epsilon)\log H_{k}-\frac{q_{k}}{2n-2}.
\]
and further combination with \eqref{eq:jordan} yields
\begin{equation}  \label{eq:umdrehen}
 L_{2n-2,2n-1}^{\ast}(q_{k}) \leq 
(\nu+\epsilon)\cdot (\tau_{k}+\tilde{\epsilon}q_{k})-\frac{q_{k}}{2n-2},
\end{equation}
where 
\[
\tau_{k}=\frac{q_{k}}{2n-2}+
q_{k}\cdot\frac{2n-2-\widehat{w}_{n}(\zeta)}{(2n-2)(1+\widehat{w}_{n}(\zeta))}
=q_{k}\cdot \frac{2n-1}{(2n-2)(1+\widehat{w}_{n}(\zeta))}.
\]
The bounds in the right hand sides of \eqref{eq:hand} and \eqref{eq:umdrehen} depend 
on $q_{k}$ and $\widehat{w}_{n}(\zeta)$ only.
Comparison of these two estimates and some
rearrangements yield the estimate
\[
(2n-1)\cdot(\widehat{w}_{n}(\zeta)^{2}-3(n-1)\widehat{w}_{n}(\zeta)+2n^{2}-4n+1)
+\epsilon^{\prime}\Phi(\widehat{w}_{n}(\zeta)) \leq 0,
\]
where $\epsilon^{\prime}$ is again some variation of $\epsilon$ which tends to $0$ as $\epsilon$ does 
and $\Phi(\widehat{w}_{n}(\zeta))$ is some bounded expression when $\widehat{w}_{n}(\zeta)$ is bounded. In the calculation we used
that $q_{k}$ cancels
out since we may treat $O(1)$ as $o(q_{k})$, and 
incorporated $\widehat{w}_{n}(\zeta)-(2n-2)>0$ by assumption.
Since we may take $\epsilon$ and thus $\epsilon^{\prime}$ arbitrarily small, we see that the larger root
of the quadratic polynomial, which is an upper 
bound for $\widehat{w}_{n}(\zeta)$, will be arbitrarily 
close to $\theta_{n}$. This contradicts 
our assumption $\widehat{w}_{n}(\zeta)>\theta_{n}$. 
Thus the proof of case 1 is finished.

Case 2: For all large $k$ the span of $\mathscr{R}_{k}$ is not the entire space
of polynomials of degree at most $2n-2$. Since $\mathscr{P}_{k}$ is a hyperspace
this means that the span of $\mathscr{R}_{k}$ coincides with the span of $\mathscr{P}_{k}$,
in fact the span of all $\mathscr{P}_{k}$ and $\mathscr{R}_{k}$ (or their union)
must be a constant hyperspace of $\mathbb{R}^{2n-1}$ for all $k\geq k_{0}$.
Suppose this is the case. For $k\geq 2$ and
$q_{k}$ defined as in case 1, consider the intervals $I_{k}:=[q_{k-1},q_{k}]$.
Recall \eqref{eq:hehe}, for which we did not use any linear independence arguments of case 1.
Moreover recall we showed in the proof of case 1 that our assumption $\widehat{w}_{n}(\zeta)>\theta_{n}$ implies
\begin{equation} \label{eq:casa1}
\max_{P\in\mathscr{P}_{k}} L_{P}^{\ast}(q_{k})=  L_{P_{k}}^{\ast}(q_{k})< L_{2n-2,2n-1}^{\ast}(q_{k}).
\end{equation}
On the other hand the function $L_{P_{k}}^{\ast}$ decays in the interval $I_{k}$ (in fact in $[0,q_{k}]$)
with slope $-1/(2n-2)$, and at $q=q_{k}$ it meets the rising phase of $L_{P_{k-1}}^{\ast}$. More generally
the functions $L_{P}^{\ast}$ for $P\in\mathscr{V}_{k}$ decay in $[0,q_{k}]$ with slope $-1/(2n-2)$
and it obviously follows that
\[
\max_{P\in\mathscr{P}_{k}} L_{P}^{\ast}(q)= L_{P_{k}}^{\ast}(q), \qquad q\in[0,q_{k}].
\]
In fact all values $L_{P}^{\ast}(q)$ for $P\in\mathscr{P}_{k}$ coincide in $[0,q_{k}]$ as their heights are equal.
Hence, as the slope of $L_{2n-2,2n-1}^{\ast}$ cannot be smaller than $-1/(2n-2)$, 
the estimate \eqref{eq:casa1} implies that  we have $L_{P}^{\ast}(q)< L_{2n-2,2n-1}^{\ast}(q)$
for all $P\in\mathscr{P}_{k}$ in the entire interval $I_{k}\ni{q}$.
On the other hand we have shown that the set $\mathscr{P}_{k}=\mathscr{V}_{k-1}\cup \mathscr{V}_{k}$ 
spans a $(2n-2)$-dimensional hyperspace and thus
\begin{equation}  \label{eq:casa2}
L_{2n-2,2n-2}^{\ast}(q)\leq \max_{P\in\mathscr{P}_{k}} L_{P}^{\ast}(q)= L_{P_{k}}^{\ast}(q), \qquad q\in I_{k}.
\end{equation}
Combining \eqref{eq:casa1} and \eqref{eq:casa2}, 
we infer the strict inequality
$L_{2n-2,2n-2}^{\ast}(q)<L_{2n-2,2n-1}^{\ast}(q)$ for all $q\in{I_{k}}$ when $k$ is large. 
Since this argument holds for all large $k$ and $\cup_{j\geq k} I_{j}=[q_{k-1},\infty]$, we derive
\[
L_{2n-2,2n-2}^{\ast}(q)<L_{2n-2,2n-1}^{\ast}(q), \qquad q\geq \tilde{q}.
\]
This contradicts~\cite[Theorem~1.1]{ss}
which directly implies that $L_{m,j}^{\ast}(q)=L_{m,j+1}^{\ast}(q)$ has arbitrarily large solutions $q$ for
any integer pair $(m,j)$ with $m\geq 1$ and $1\leq j\leq m$, unless $\zeta$ is algebraic of degree at most $m$. 
Thus the assumption of case 2 cannot occur for 
transcendental $\zeta$
when $\widehat{w}_{n}(\zeta)>\theta_{n}$, 
and this case is proved as well. For algebraic $\zeta$ we
know the better bounds from \eqref{eq:alg} anyway.
\end{proof}

\begin{remark} \label{dadpr}
As quoted in the proof of Lemma~\ref{lehmer},
for any transcendental $\zeta$ there exist
infinitely many $k$ such that two successive best approximation polynomials
$P_{k-1}^{n,\zeta}$ and $P_{k}^{n,\zeta}$ have no common factor. However,
for our method in case 2 to work, we had to guarantee this property
for all large $k$. Thus \eqref{eq:eqq} was needed. On the
other hand, the
fact that the degrees of $P_{k-1}^{n,\zeta}$ and $P_{k}^{n,\zeta}$ are precisely $n$ 
can be avoided by the argument used in the proof
of $\Psi(n)\leq 2n-1$ in Lemma~\ref{lehmer}.
\end{remark}

\begin{remark} \label{dierem}
The method of the proof can be used to infer
\begin{equation} \label{eq:nichtgut}
\frac{w_{n}(\zeta)}{\widehat{w}_{n}(\zeta)}\geq \widehat{w}_{n}(\zeta)-2n+3, \qquad \text{if} \;\; \widehat{w}_{n}(\zeta)>2n-2.
\end{equation}
For $n=2$ this affirms the estimate $w_{2}(\zeta)\geq \widehat{w}_{2}(\widehat{w}_{2}(\zeta)-1)$ known by Jarn\'ik.
However, for $n\geq 3$ it can be checked that for 
\eqref{eq:nichtgut} to be better than \eqref{eq:schs}
requires
$\widehat{w}_{n}(\zeta)$ to be larger than the upper bound 
in \eqref{eq:frwe}. Hence, in contrast
to Theorem~\ref{gutsatz}, no new information
on the quotient $w_{n}(\zeta)/\widehat{w}_{n}(\zeta)$
is obtained.
\end{remark}

In dimension $m=2n-2$ as in the proof, we cannot
expect something better than 
$\widehat{w}_{n}(\zeta)\leq 2n-2+o(1)$ with 
positive remainder term.
If the essential arguments of the proof
can be transferred to the situation of some dimension $m<2n-2$,
with $m$ not too small compared to $n$,
we would expect to obtain some better bound, as in
Theorem~\ref{gutsatz} and Theorem~\ref{cumbersome}.
For example when $m=3\frac{n}{2}-1$, the expected
bounds turn out be just as in \eqref{eq:rhs}.
The difficulty when choosing $m<2n-2$
is to guarantee the linear independence of a sufficiently large subset
of the polynomials defined similarly as $\mathscr{R}_{k}$.
This was an essential step to obtain reasonable bounds with 
Minkowski's Second Convex Body
Theorem (or parametric geometry of numbers). 
Indeed,
if $m<2n-2$, the codimension of the accordingly
modified set $\mathscr{P}_{k}$ 
in the space of polynomials at most $m$ 
turns out to be $2n-1-m>1$, and the argument of case 2
from the proof of Theorem~\ref{thm} fails and yet we do not know
how to modify it. 
In the proof of case 1 of Theorem~\ref{thm}, we
showed that if $\Psi(n,\zeta)\leq 2n-2$, then
$\widehat{w}_{n}(\zeta)$ cannot exceed
$\theta_{n}$. Some additional argument was needed
in case 2 to derive the same estimate from $\Psi(n,\zeta)=2n-1$.
Concerning Theorem~\ref{gutsatz}, its main assumption is
essentially equivalent to
$\Psi(n,\zeta)=3\frac{n}{2}-1$ and resolves the mentioned
problems from case 2. 
The additional assumption \eqref{eq:burt} enters
primarily to guarantee \eqref{eq:puschel}. If otherwise \eqref{eq:burt} fails, that is
$w_{n}(\zeta)=w_{n-1}(\zeta)$,
we derive \eqref{eq:wehave} easily from \eqref{eq:mundn}
for $n\geq 4$, as in 
the deduction of~\cite[Theorem~2.1]{buschl}.
In the proof we will denote by $\epsilon_{i}$
certain small variations of $\epsilon$, and state 
beforehand that any of them tends to $0$ as $\epsilon$ does.
 
\begin{proof}[Proof of Theorem~\ref{gutsatz}]
Suppose $\widehat{w}_{n}(\zeta)>\sigma_{n}$ holds
for some real $\zeta$. We can again assume $\zeta$ is transcendental
in view of \eqref{eq:alg}.
Let $m=\frac{3n}{2}-1$ and consider the 
combined Schmidt-Summerer graph associated to 
$(m,\zeta)$. 
For $(P_{k})_{k\geq 1}=(P_{k}^{n,\zeta})_{k\geq 1}$ the sequence of
best approximation polynomials associated to $(n,\zeta)$,
we again define $(q_{k})_{k\geq 1}$ the sequence of points where $L_{P_{k-1}}^{\ast}$ meets $L_{P_{k}}^{\ast}$,
that is $q_{k}$ is implicitly defined by
\[
L_{P_{k-1}}^{\ast}(q_{k})=L_{P_{k}}^{\ast}(q_{k})=
\log H(P_{k})-\frac{q_{k}}{m}= \log \vert P_{k-1}(\zeta)\vert + q_{k}.
\]
Let further
\[
V_{l,j}(T)= T^{j}P_{l}(T), \qquad\qquad l\geq 1, \quad 0\leq j\leq m-n,
\]
and derive the sets $\mathscr{V}_{l},\mathscr{P}_{l}$ 
and $\mathscr{R}_{l}$ very
similarly as in the proof of Theorem~\ref{thm}, 
which are now of cardinality 
$m-n+1, 2(m-n+1)$ and $3(m-n+1)$ respectively.
For the same reason as in Theorem~\ref{thm} we have
\eqref{eq:glei}.
Again without loss of generality let $\zeta\in(0,1)$, 
such that for any $l\geq 1$
the value $\vert P_{l}(\zeta)\vert$ maximizes $\vert P(\zeta)\vert$ among $P\in\mathscr{V}_{l}$.
Now since \eqref{eq:ingung} holds by assumption,
and by the choice of $m$ 
we have $\sharp \mathscr{R}_{k}=3(m-n+1)=m+1$,
the set $\mathscr{R}_{k}$ is linearly independent and
spans the space of polynomials of degree at most $m$. In particular $\mathscr{P}_{k}\subseteq \mathscr{R}_{k}$
is linearly independent and thus,
very similar to \eqref{eq:jetzaber}, from
\eqref{eq:umrechnen} and \eqref{eq:glei} we obtain
\begin{equation} \label{eq:jetztaber}
L_{m,2(m-n+1)}^{\ast}(q_{k}) \leq L_{P_{k}}^{\ast}(q_{k})\leq 
q_{k}
\cdot \frac{m-\widehat{w}_{n}(\zeta)}{m(1+\widehat{w}_{n}(\zeta))}+\epsilon_{0}q_{k}.
\end{equation}
With \eqref{eq:gimme} applied for $j=2(m-n+1)$ we conclude            
\begin{equation}  \label{eq:handa}
L_{m,m+1}^{\ast}(q_{k}) \geq 
\left(-\frac{2(m-n+1)}{m+1-2(m-n+1)}\right)\cdot
q_{k}\cdot \frac{m-\widehat{w}_{n}(\zeta)}{m(1+\widehat{w}_{n}(\zeta))}+\epsilon_{1}q_{k}+O(1).
\end{equation}
Now again let $H_{l}=H(P_{l})$ for any integer $l\geq 1$.
We again infer
\[
\log H_{k}-\frac{q_{k}}{m}= L_{P_{k}}^{\ast}(q_{k}),
\]
and together with \eqref{eq:jetztaber} we derive
\begin{equation} \label{eq:jordansche}
\log H_{k}\leq \frac{q_{k}}{m}+
q_{k}\cdot \frac{m-\widehat{w}_{n}(\zeta)}{m(1+\widehat{w}_{n}(\zeta))}+\epsilon_{2}q_{k}
=q_{k}\left(\frac{m+1}{m(1+\widehat{w}_{n}(\zeta))}+\epsilon_{2}\right).
\end{equation}
Let $R\in\mathscr{V}_{k+1}$ arbitrary and put $r_{k+1}$ the local minimum of $L_{R}^{\ast}$.
Now observe that $\sigma_{n}$ is larger than the 
dimension $m=\frac{3n}{2}-1$.
Hence the right hand side of \eqref{eq:jetztaber} is negative, and 
moreover in view of \eqref{eq:umrechnen} at the local 
minimum $r_{k+1}$ of $L_{R}^{\ast}$
we also have $L_{R}^{\ast}(r_{k+1})<0$. 
On the other hand again $L_{R}^{\ast}(q_{k})>0$, as we carry out. 
Since by their definition for any $P,Q\in\mathscr{V}_{l}$ 
we have $H(P)=H(Q)$ and the evaluations $P(\zeta)$ and $Q(\zeta)$
differ only by a multiplicative constant, it follows from \eqref{eq:piscibus}
that $\vert L_{P}^{\ast}(q)-L_{Q}^{\ast}(q)\vert \ll 1$ uniformly on $q\in(0,\infty)$.
Thus if we had $L_{R}^{\ast}(q_{k})\leq 0$ then by the linear independence of $\mathscr{R}_{k}$
we have all $L_{m,j}^{\ast}(q_{k})\ll 1$ for $1\leq j\leq m+1$, and the 
first $2(m-n+1)$ values $1\leq j\leq 2(m-n+1)$
are even bounded above by $(-c+o(1))q_{k}$ for some fixed $c>0$
independent from $k$
in view of \eqref{eq:jetztaber}. Hence the sum
of $L_{m,j}^{\ast}(q_{k})$ over $1\leq j\leq m$ is at most
$(-c+o(1))q_{k}+O(1)$,
contradicting \eqref{eq:lsumme} for large $k$.
Hence the claim is shown. 
We conclude that $r_{k+1}>q_{k}$ and $L_{R}^{\ast}$ still decays at $q_{k}$, and hence
\[
\log H_{k+1}-\frac{q_{k}}{m}= L_{R}^{\ast}(q_{k}), \qquad R\in\mathscr{V}_{k+1}.
\]
For now assume $w_{n}(\zeta)>w_{n-1}(\zeta)$. Then
we may apply \eqref{eq:puschel},
and as in the proof Theorem~\ref{thm} 
with Lemma~\ref{ehklar} we infer \eqref{eq:jojo}. Hence
again since $\mathscr{R}_{k}$ spans the entire space of polynomials of 
degree at most $m$ from \eqref{eq:jordansche} we derive
\begin{equation}  \label{eq:umdrehena}
 L_{m,m+1}^{\ast}(q_{k}) \leq L_{R}^{\ast}(q_{k})\leq 
(\nu+\epsilon)\cdot (\tau_{k}+\epsilon_{2}q_{k})-\frac{q_{k}}{m},
\end{equation}
where 
\[
\tau_{k}
=q_{k}\cdot \frac{m+1}{m(1+\widehat{w}_{n}(\zeta))}.
\]
We combine \eqref{eq:handa} and \eqref{eq:umdrehena}, 
which leads after some computation to
\[
(1+m)\cdot(\widehat{w}_{n}(\zeta)^{2}+(n-2m-1)\widehat{w}_{n}(\zeta)+(-4n^{2}-1+3mn-m+4n))+
\epsilon_{3}\Phi(\widehat{w}_{n}(\zeta))\leq 0,
\]
where again $\Phi$ is bounded. We insert $m=\frac{3}{2}n-1$ and obtain
\begin{equation} \label{eq:inss}
\frac{3n}{4}\cdot\left(2\widehat{w}_{n}(\zeta)^{2}+(2-4n)\widehat{w}_{n}(\zeta)+n^{2}-n\right)+\epsilon_{3}\Phi(\widehat{w}_{n}(\zeta))\leq 0.
\end{equation}
For $\epsilon_{3}=0$, the quadratic inequality
is satisfied precisely
for $\widehat{w}_{n}(\zeta)\leq \sigma_{n}$ 
with $\sigma_{n}$ in \eqref{eq:rhs},
such that our assumption of
strict inequality $\widehat{w}_{n}(\zeta)>\sigma_{n}$ 
yields a contradiction if we start with $\epsilon$ 
sufficiently small.

Now assume $w_{n}(\zeta)=w_{n-1}(\zeta)$.
Then 
$\widehat{w}_{n}(\zeta) \leq 2n-2$ by \eqref{eq:mundn} 
applied with $n_{1}=n, n_{2}=n-1$, as in~\cite[Theorem~2.1]{buschl}. 
Hence
\[
\widehat{w}_{n}(\zeta) \leq \max \left\{ 2n-2, \frac{2n-1+\sqrt{2n^2-2n+1}}{2} \right\}.
\]
It can be readily checked that for $n=4$ both bounds coincide and for $n\geq 6$
the bound $2n-2$ is larger.

Finally we sketch how to derive \eqref{eq:derive}. For
any $\widehat{w}_{n}(\zeta)>m$ we can 
proceed as in the proof to obtain \eqref{eq:handa}. Moreover,
Lemma~\ref{ehklar} again yields \eqref{eq:umdrehena} where in place
of $\nu$ from \eqref{eq:jojo}
we have to keep the expression 
$w_{n}(\zeta)/\widehat{w}_{n}(\zeta)$. 
Combining these two estimates yields the claim, we skip the
computation.
\end{proof}

The main ideas of the proof of Theorem~\ref{cumbersome}
are again very similar. For convenience we state two
easy propositions first. 
The first one, as well as its proof, is closely related
to Lemma~\ref{lehmer} and Lemma~\ref{lainyweg}.

\begin{proposition} \label{babi}
Let $n\geq 2$
and integer and $\zeta$ a transcendental real number.
Use the notation of Definition~\ref{dadp}. 
For any integer $m\geq n$ define the set
\[
\mathscr{C}_{k,m}^{n,\zeta}\; :=\; \mathscr{A}_{k-1,m}^{n,\zeta}\cup \mathscr{A}_{k,m}^{n,\zeta}\; \subseteq \; \mathscr{B}_{k,m}^{n,\zeta}, \qquad \qquad k\geq 2.
\]
Then for all large $k$ set $\mathscr{C}_{k,m}^{n,\zeta}$ spans a vectorspace of dimension
at least $m-n+2$.
If
$\widetilde{\Psi}(n,\zeta)\leq m\leq 2n-1$, then for infinitely many $k$ the 
set $\mathscr{C}_{k,m}^{n,\zeta}$
spans a space of dimension at least
$2(m-n+1)$. 
\end{proposition}

\begin{proof}
Without loss of generality assume $d_{k-1}\leq d_{k}$
for $d_{k-1}$ and $d_{k}$ the degrees of
$P_{k-1}^{n,\zeta}$ and $P_{k}^{n,\zeta}$ respectively, 
otherwise alter the set $\widetilde{\Omega}_{k}$ below accordingly.
For the first claim, it suffices to consider the set
\begin{equation} \label{eq:sieheoben}
\widetilde{\Omega}_{k}:= 
\left\{ P_{k-1}^{n,\zeta},\; P_{k}^{n,\zeta},\; TP_{k}^{n,\zeta},
\; T^{2}P_{k}^{n,\zeta},\ldots,\;
T^{m-d_{k}}P_{k}^{n,\zeta}\right\}\subseteq \mathscr{C}_{k,m}^{n,\zeta}.
\end{equation}
Since $P_{k-1}^{n,\zeta},P_{k}^{n,\zeta}$ are linearly independent,
it is easily seen that $\widetilde{\Omega}_{k}$ is
as well, and has cardinality $\sharp \widetilde{\Omega}_{k}=m-d_{k}+2\geq m-n+2$.
We need to show the second claim.
By assumption $m\geq \widetilde{\Psi}(n,\zeta)$ we may assume
that $P_{k-1}^{n,\zeta}$ and $P_{k}^{n,\zeta}$ have no common
factor for certain arbitrarily large $k$. For such $k$,
as pointed out in the proof of Lemma~\ref{lehmer},
the set $\Omega_{k}$ in \eqref{eq:lemmer}
of polynomials of degree at most
$d_{k-1}+d_{k}-1$
is linearly independent. We now distinguish two
cases. Case 1: $m\leq d_{k-1}+d_{k}-1$. Then 
\[
\left\{  P_{k-1}^{n,\zeta},\; TP_{k-1}^{n,\zeta},\ldots,\; 
T^{m-d_{k}}P_{k-1}^{n,\zeta},\;
P_{k}^{n,\zeta},\; TP_{k}^{n,\zeta},\ldots,\;
T^{m-d_{k-1}}P_{k}^{n,\zeta}\right\}
\]
is a subset of $\Omega_{k}$, consisting
of polynomials of degree at most $m$. It is
obviously linearly independent (as $\Omega_{k}$ is) and has cardinality 
$(m-d_{k}+1)+(m-d_{k-1}+1)\geq 2(m-n+1)$ since
$d_{i}\leq n$. Case 2: $m>d_{k-1}+d_{k}-1$.
Then consider the set
\[
\Omega_{k}^{\prime}:=\left\{  P_{k-1}^{n,\zeta},\; TP_{k-1}^{n,\zeta},\ldots,\; 
T^{d_{k}-1}P_{k-1}^{n,\zeta},\;
P_{k}^{n,\zeta},\; TP_{k}^{n,\zeta},
\ldots,\; T^{m-d_{k}}P_{k}^{n,\zeta}\right\}
\]
derived from $\Omega_{k}$ by adding certain polynomials of higher degree. Similar to \eqref{eq:sieheoben},
it is easy to see that $\Omega_{k}^{\prime}$ is linearly independent again, as any polynomial in 
$\Omega_{k}^{\prime}\setminus \Omega_{k}$ has strictly larger degree
than any polynomial in $\Omega_{k}$ and the new degrees are all different.
Moreover $\Omega_{k}^{\prime}$ has cardinality
$d_{k}+(m-d_{k}+1)=m+1\geq 2(m-n+1)$ since $m\leq 2n-1$ by
assumption.
\end{proof}

\begin{remark}
The bound $m-n+2$ of the first claim is sharp in case 
$P_{k-1}^{n,\zeta},P_{k}^{n,\zeta}$ are of degree $n$ and
have a common factor of maximum degree $n-1$. Prescribing an upper bound $d$ on the common factor
results in lower dimension estimates in terms of $d$ in the range
between $m-n+2$ and $2(m-n+1)$.
\end{remark}

Our second preparatory result is
about extensions of linearly independent sets to bases,
and almost a triviality.

\begin{proposition} \label{t}
Let $l\leq s$ be integers and $v_{1},\ldots,v_{s}$ be vectors 
that span
a vectorspace $\mathscr{S}$ of dimension $t\leq s$. Assume 
$v_{1},\ldots,v_{l}$ are linearly independent. Then
we can find $t-l$ vectors $w_{1},\ldots,w_{t-l}$
among $v_{l+1},\ldots,v_{t}$
with the property that $v_{1},\ldots,v_{l},w_{1},\ldots,w_{t-l}$
span $\mathscr{S}$.
\end{proposition}

\begin{proof}
Consider any maximum linear independent set containing
$v_{1},\ldots,v_{l}$ by adding
some remaining vectors $v_{j}, l<j\leq k$.
Clearly this set has cardinality at most
$t$, as it is linearly independent and
its span is contained in $\mathscr{S}$. On the other hand,
it must have cardinality at least $t$, otherwise
we could add some element not in the span to increase the dimension. 
After some relabeling the arising set has the desired property.
\end{proof}

Now we turn to the proof of Theorem~\ref{cumbersome}.

\begin{proof}[Proof of Theorem~\ref{cumbersome}]
We start with \eqref{eq:jor}. 
For simplicity let $m=\widetilde{\Psi}(n,\zeta)$.
By definition of $\widetilde{\Psi}(n,\zeta)$, there exist
arbitrarily large $k$ for which
two successive $P_{k-1}^{n,\zeta},
P_{k}^{n,\zeta}$ are coprime and $\mathscr{B}_{k,m}^{n,\zeta}$
defined via \eqref{eq:zingung} and \eqref{eq:zwanzig}
span the space of polynomials of degree at most $m$.
Fix such $k$ large enough. By Proposition~\ref{babi} the space
spanned by $\mathscr{C}_{k,m}^{n,\zeta}\subseteq \mathscr{B}_{k,m}^{n,\zeta}$ has dimension 
at least $2(m-n+1)$. Keep $\mathscr{V}_{l},\mathscr{P}_{l}, \mathscr{R}_{l}$ for $l\geq 1$ from
the proof of Theorem~\ref{gutsatz}.
By assumption
and Proposition~\ref{t} applied to $s=3(m-n+1),l=2(m-n+1)$ 
and the vectors
$\{v_{1},\ldots,v_{s}\}=\mathscr{R}_{k}$ and
$\{v_{1},\ldots,v_{l}\}=\mathscr{P}_{k}$,
there exist $m+1-2(m-n+1)$ remaining polynomials
among $\mathscr{V}_{k+1}$ which together with 
$\mathscr{P}_{k}$ spans the space of polynomials of
degree at most $m$. 
Moreover, 
\eqref{eq:ingung} holds again by assumption, 
and we can again apply Lemma~\ref{ehklar}.
Again combining these arguments yields \eqref{eq:umdrehena}.
We then proceed as in the proof of Theorem~\ref{gutsatz} up
to the point where we inserted $m$ to obtain \eqref{eq:inss}.
We solve the general quadratic equation in terms of $m,n$ and 
obtain the claimed bound \eqref{eq:jor}.

Now we show \eqref{eq:fehlt}. From the first claim of
Proposition~\ref{babi} the set $\mathscr{C}_{k,m}^{n,\zeta}$
spans a space of dimension at least $m-n+2$. 
We essentially proceed as in
the proof of Theorem~\ref{gutsatz} again, 
but taking into account the lower cardinality in place of
\eqref{eq:jetztaber} and \eqref{eq:handa} we obtain
\[
L_{m,m-n+2}^{\ast}(q_{k}) \leq L_{P_{k}}^{\ast}(q_{k})\leq 
q_{k}
\cdot \frac{m-\widehat{w}_{n}(\zeta)}{m(1+\widehat{w}_{n}(\zeta))}+\epsilon_{0}q_{k}.
\]
and consequently
\[
L_{m,m+1}^{\ast}(q_{k}) \geq 
\left(-\frac{m-n+2}{m+1-(m-n+2)}\right)\cdot
q_{k}\cdot \frac{m-\widehat{w}_{n}(\zeta)}{m(1+\widehat{w}_{n}(\zeta))}+\epsilon_{1}q_{k}+O(1).
\]
On the other hand we obtain \eqref{eq:umdrehena}
precisely as above. Again combination and some calculation
yields the bound in \eqref{eq:fehlt}. 
\end{proof}

%
%

%
%

%





\begin{thebibliography}{99} 

\bibitem{baks} A. Baker and W.M. Schmidt, Diophantine approximation and Hausdorff dimension,
{\em Proc. London Math. Soc.} 21 (1970), 1--11.





\bibitem{bernik} V. I. Bernik, Application of the Hausdorff dimension in the theory of Diophantine approximations,
{\em Acta Arith.} 42 (1983), 219--253 (in Russian). English translation in Amer. Math. Soc. Transl. 140 (1988), 15--44.


\bibitem{bugbuch} Y. Bugeaud. Approximation by Algebraic Numbers,
{\em Cambridge Tracts in Mathematics} 160 (2004), Cambridge University Press.

\bibitem{buglau} Y. Bugeaud and M. Laurent. Exponents of Diophantine approximation
and Sturmian continued fractions. {\em Ann. Inst. Fourier (Grenoble)} 55 (2005), no. 3, 773--804.

\bibitem{buschl} Y. Bugeaud and J. Schleischitz. On uniform approximation to real numbers.
{\em Acta Arith.} 175 (2016), no. 3, 255--268. 

\bibitem{davs} H. Davenport and W. M. Schmidt. 
Approximation to real numbers by quadratic
irrationals. {\em Acta Arith.} 13 (1967), 169--176.

\bibitem{davsh} H. Davenport and W. M. Schmidt. 
Approximation to real numbers by algebraic
integers. {\em Acta Arith.} 15 (1969), 393--416.

\bibitem{f1} S. Fischler. Spectres pour l'approximation d'un nombre r\'eel et de son carr\'e.
{\em C.R. Acad. Sci. Paris} 339 (2004), 679--682.

\bibitem{f2} S. Fischler. Palindromic prefixes and episturmian words. 
{\em J. Comb. Theory Ser. A} 113 (2006), 1281--1304.

\bibitem{f3} S. Fischler. Palindromic prefixes and Diophantine approximation. {\em Monatsh. Math.} 151 (2007), 11--37.

\bibitem{jarid} V. Jarn\'ik. Zum Khintchineschen \"Ubertragungssatz", 
{\em Trav. Inst. Math. Tbilissi} 3 (1938), 193--212.


\bibitem{maler} K. Mahler. Zur  Approximation  der  Exponentialfunktionen  und  des  Logarithmus. I,
II, {\em J. reine angew.  Math.}  166 (1932), 118--150.


\bibitem{mamo} A. Marnat, N. Moshchevitin.
An optimal bound for the ratio between ordinary and uniform exponents of Diophantine approximation. {\em arXiv: 1802.03081}.

\bibitem{mink} H. Minkowski. Geometrie der Zahlen. {\em Teubner, Leipzig 1910.}

\bibitem{mosh} N. G. Moshchevitin. Best Diophantine approximations: the phenomenon of degenerate 
dimension. {\em London Math. Soc. Lecture Note Ser.} 338, Cambridge Univ. Press, Cambridge 2007, 158--182.

\bibitem{mo1} N. G. Moshchevitin. Exponents for three-dimensional
simultaneous approximation. {\em Czechoslovak Math. J.} 62 (2011), no. 1, 325--362.



\bibitem{royyy} D. Roy. Approximation to real numbers by cubic algebraic integers I. 
{\em Proc. London Math. Soc.} 88 (2004), 42--62.

\bibitem{rroy} D. Roy. Diophantine approximation in small degree. {\em Number Theory, 269--285, CRM Proc.
Lecture Notes} 36, Amer. Math. Soc., Providence, RI, 2004.
  

\bibitem{royexp} D. Roy. On two exponents of approximation related to a real number and its square.
{\em Canad. J. Math.} 59 (2007), 211--224. 

\bibitem{royann} D. Roy. On Schmidt-Summerer parametric geometry of numbers. {\em Ann. Math.} 182 (2015), 739--786.

\bibitem{ich1} J. Schleischitz. Diophantine approximation and special Liouville 
numbers. {\em Comm. Math.} 21 (2013), 39--76.

\bibitem{j2} J. Schleischitz. Two estimates concerning classical Diophantine approximation constants,
{\em Publ. Math. Debrecen} 84/3-4 (2014), 415--437.


\bibitem{ichindag} J. Schleischitz. On uniform approximation to successive powers
of a real number. {\em Indag. Math.} 28 (2017), no. 2, 406--423.

\bibitem{ichglasgow} J. Schleischitz. Some notes on the regular graph defined by Schmidt and Summerer
and uniform approximation. {\em JP J. Algebra Number Theory Appl.} 39 (2017), no. 2, 115--150.




\bibitem{ichcomments} J. Schleischitz. On the discrepancy between
best and uniform approximation. {\em appeared online 
in Funct. Approx. Comment. Math., March 28th 2018}

\bibitem{ss} W.M. Schmidt and L. Summerer. Parametric geometry of numbers and applications. 
{\em Acta Arith.} 140 (2009), no. 1,  67--91.  


\bibitem{ssch} W.M. Schmidt and L. Summerer. 
Diophantine approximation and parametric geometry of numbers. 
{\em Monatsh. Math.} 169 (2013), 51--104. 

\bibitem{sums}  W.M. Schmidt and L. Summerer. 
Simultaneous approximation to three numbers.
{\em Mosc. J. Comb. Number Theory} 3 (2013), 84--107.


\bibitem{wirsing} E. Wirsing. Approximation mit algebraischen Zahlen beschr\"ankten Grades.
{\em J. Reine Angew. Math.}  206 (1961), 67--77.

\end{thebibliography}
\end{document}